\pdfoutput=1
\documentclass[11pt]{amsart}
\usepackage[T1]{fontenc}
\usepackage{lmodern}
\usepackage[letterpaper,margin=1in]{geometry}
\usepackage{amsmath,amssymb,mathtools,microtype}
\usepackage{enumitem,tikz-cd,needspace,booktabs,array,tabularx}
\usetikzlibrary{arrows.meta,fit,backgrounds,shapes.geometric}
\usepackage[hidelinks]{hyperref}
\newtheorem{theorem}{Theorem}[section]
\newtheorem{mainthm}[theorem]{Theorem}
\newtheorem{mainconj}{Conjecture}

\newtheorem{proposition}[theorem]{Proposition}
\newtheorem{lemma}[theorem]{Lemma}
\newtheorem{corollary}[theorem]{Corollary}
\theoremstyle{definition}
\newtheorem{definition}[theorem]{Definition}
\newtheorem{example}[theorem]{Example}
\theoremstyle{remark}
\newtheorem{remark}[theorem]{Remark}
\numberwithin{equation}{section}
\newcommand{\Z}{\mathbb Z}
\newcommand{\F}{\mathbb F}
\newcommand{\C}{\mathbb C}
\newcommand{\R}{\mathbb R}
\newcommand{\one}{\mathbf 1}
\newcommand{\id}{\mathrm{id}}
\newcommand{\Ext}{\operatorname{Ext}}
\newcommand{\Hom}{\operatorname{Hom}}
\newcommand{\End}{\operatorname{End}}

\newcommand{\Thick}{\operatorname{Thick}}
\newcommand{\cofib}{\operatorname{cofib}}

\newcommand{\Mod}{\operatorname{Mod}}
\newcommand{\SH}{\mathrm{SH}}
\newcommand{\cell}{\mathrm{cell}}

\newcommand{\D}{\mathcal D}
\tikzset{cell/.style={circle,draw,fill=white,inner sep=1.8pt},
  labeltext/.style={font=\scriptsize,fill=white,inner sep=1.5pt},
  block/.style={draw,rounded corners=3pt,inner sep=9pt},
  map/.style={-{Stealth[length=2mm]},semithick}}
\tikzset{
  motivic cell/.style={draw,ellipse,fill=white,inner xsep=2pt,
    inner ysep=1.5pt,font=\scriptsize,minimum width=11mm,minimum height=6mm},
  cell label/.style={font=\scriptsize,inner sep=1.5pt,outer sep=1pt},
  cell map/.style={-{Stealth[length=1.8mm,width=1.3mm]},semithick}
}

\usepackage{capt-of,placeins}
\title[Generating hypotheses along the motivic deformation]
{Generating hypotheses along the motivic deformation}
\author{Sihao Ma}
\address{UCLA Department of Mathematics, 520 Portola Plaza, Los Angeles, CA 90095-1555}
\email{masihao@math.ucla.edu}
\author{Zhouli Xu}
\address{UCLA Department of Mathematics, 520 Portola Plaza, Los Angeles, CA 90095-1555}
\email{xuzhouli@ucla.edu}
\date{}
\hypersetup{pdftitle={Generating hypotheses along the motivic deformation},
pdfauthor={Sihao Ma and Zhouli Xu},
pdfkeywords={Generating hypothesis, motivic spectra, BP comodules, formal groups}}
\begin{document}
\begin{abstract}
At every prime $p$, we disprove the algebraic generating hypothesis for compact objects in Hovey's stable category of $BP_*BP$-comodules and its local analog at every positive height, answering questions of Barthel and Heard. In both settings, for every $r,n\geq1$, we construct a ghost whose first $n$ composition powers are all nonzero and have exact additive order $p^r$. The same construction applies to the stable category of $E_*E$-comodules for even $p$-local Landweber exact theories $E$ that are not rational.

Through the motivic deformation of Gheorghe--Wang--Xu, these examples give $C\tau$-linear counterexamples to the cellular $\C$-motivic generating hypothesis, with the same order and composition properties. We also construct nonzero non-$C\tau$-linear ghosts between $C\tau$-modules and a nonzero ghost on a motivic spectrum with noncontractible Betti realization.
\end{abstract}
\maketitle
\enlargethispage{4pt}
\raggedbottom
\setcounter{tocdepth}{1}
\tableofcontents
\section{Introduction}\label{sec:intro}

\subsection{Generating hypotheses and the motivic deformation}\label{subsec:introduction}

The generating hypothesis of Freyd~\cite{Freyd} states that a map $f:X\to Y$ between finite spectra is null-homotopic if $\pi_n(f)=0$ for every integer $n$. Equivalently, stable homotopy groups define a faithful functor on the finite stable homotopy category. This conjecture remains open.

Its significance extends beyond detecting zero maps. Freyd~\cite[Proposition~9.7]{Freyd} proved that faithfulness would imply fullness: writing $S$ for the classical sphere, the natural map
\[
 [X,Y]\longrightarrow\Hom_{\pi_*S}(\pi_*X,\pi_*Y)
\]
would be an isomorphism for every pair of finite spectra. In particular, the graded $\pi_*S$-module $\pi_*X$ would determine the stable homotopy type of $X$. In the $p$-complete formulation, the hypothesis would also make $\pi_*(S_p^{\wedge})$ self-injective in graded modules and totally incoherent~\cite{Freyd,Hovey,SS}.

Analogous questions have been settled in several algebraic settings. Lockridge~\cite[Corollary~3.8]{Lockridge} proved that, for a commutative ring $R$, homology is faithful on perfect complexes over $R$ if and only if $R$ is von Neumann regular. For a nontrivial finite $p$-group over a field of characteristic $p$, Benson, Chebolu, Christensen, and Min\'a\v{c}~\cite[Theorem~1.1]{BCCM} proved that Tate cohomology is faithful on the stable category of finite-dimensional modules exactly for the cyclic groups of order $2$ and $3$. In rational equivariant stable homotopy theory, the Greenlees--May description~\cite[Theorem~A.4]{GM} gives faithfulness for finite groups when all orbit spheres are used as detectors, whereas Bohmann~\cite[Theorem~1.7]{Bohmann} proved failure for the circle group.

\medskip
\noindent\textbf{The Question of Barthel and Heard.}
Fix a prime $p$. The $E_2$-page of the Adams--Novikov spectral sequence
\[
 E_2^{s,t}=\Ext^{s,t}_{BP_*BP}(BP_*,BP_*)
\]
gives an approximation of the $p$-primary stable homotopy groups of spheres. Its entries are derived maps between shifts of the comodule unit $BP_*$.

Quillen's theorem~\cite[Theorem~2]{Quillen} gives this approximation to stable homotopy a geometric interpretation: even graded $BP_*BP$-comodules describe quasi-coherent sheaves on the moduli stack of formal groups over $\Z_{(p)}$, and the detecting Ext groups are sheaf cohomology with twists by the line bundle of invariant differentials \cite[Corollary~2.45, Remark~3.14, and equation~(3.5)]{Goerss}. Barthel and Heard~\cite[Remark~5.15]{BH} ask for algebraic analogs of the generating hypothesis, both globally and locally at a fixed chromatic height.

Let $\D(BP_*BP)$ denote the derived category of graded $BP_*BP$-comodules, and put
\[
 \D_{\mathrm{fin}}(BP_*BP) =\Thick\{(\Sigma^{2b}BP_*)[a]:a,b\in\Z\} \subset\D(BP_*BP).
\]
Here $\Thick$ denotes thick closure, $\Sigma^t$ shifts the internal degree by $t$, and $[r]$ shifts the complex degree. We use the cochain convention specified in Section~\ref{subsec:categories}. Under the stable-to-derived comparison, this is the compact part of the even graded subcategory of Hovey's stable category of comodules $\operatorname{Stable}(BP_*BP)$ \cite[Section~6]{Hovey-comodules}; see Section~\ref{subsec:categories} for the comparison. For $h\geq1$, write $L_{K(h)}$ for algebraic $K(h)$-localization \cite[Definition~6.13]{BH}. The local category is the even graded part of $L_{K(h)}\operatorname{Stable}(BP_*BP)$, with unit
\[
 L_{K(h)}BP_*\simeq \operatorname*{holim}_{s\geq1}v_h^{-1}(BP_*/I_h^s), \qquad I_h=(p,v_1,\ldots,v_{h-1}).
\]
Here the homotopy limit is taken in $\operatorname{Stable}(BP_*BP)$. Section~\ref{subsec:local-algebraic} explains this formula and its Johnson--Wilson presentation. In either category, a map $g$ is a \emph{ghost} if
\begin{equation}\label{eq:ghost-definition}
 \Hom(\Sigma^{2b}\one,g[s])=0\qquad(s,b\in\Z),
\end{equation}
where $\one=BP_*$ or $L_{K(h)}BP_*$, respectively. The following statements formulate affirmative answers to the two questions~\cite[Remark~5.15]{BH}.

\begin{mainconj}[Algebraic Generating Hypothesis]\label{conj:algebraic-gh}
\leavevmode
\begin{enumerate}[label=\textup{(\arabic*)},leftmargin=2em]
\item Every ghost between objects of $\D_{\mathrm{fin}}(BP_*BP)$ is zero.
\item For every $h\geq1$, every ghost between compact objects of $L_{K(h)}\operatorname{Stable}(BP_*BP)$ is zero.
\end{enumerate}
\end{mainconj}

Barthel and Heard expect the local statement to fail, in analogy with the topological $K(h)$-local result of Barthel~\cite[Corollary~3.9]{Barthel-local}. We disprove both statements at every prime by the same formal-group operation.

Let $E$ be an even, $p$-local Landweber exact theory with a $p$-typical orientation $BP_*\to E_*$. Its ordinary cooperations form the flat Hopf algebroid \cite[Introduction and Corollary~2.3]{HS}
\[
 E_*E\cong E_*\otimes_{BP_*}BP_*BP\otimes_{BP_*}E_*.
\]
For $r\geq1$, multiplication by an integer $u\equiv1\pmod{p^r}$ on the formal group of $E$ over $E_*/p^r$ is a strict automorphism. It commutes with coordinate isomorphisms and hence induces a natural automorphism $\Theta_u$ of each graded $E_*E/p^r$-comodule $L$; see Section~\ref{subsec:strict-operation} for the construction.

\begin{proposition}[Ghosts from formal multiplication]\label{prop:intro-criterion}
For every graded $E_*E/p^r$-comodule $L$ and every integer $u\equiv1\pmod{p^r}$, the operation $\Theta_u$ induces the identity on $\Ext^{s,t}_{E_*E}(E_*,L)$ for all $s\geq0$ and $t\in\Z$. Consequently, $1-\Theta_u:L\to L$ is a ghost.
\end{proposition}

The proposition applies to $BP$, Johnson--Wilson theories $E(n)$, and Morava theories $E_n$ for $n\geq1$. To construct a nonzero ghost, it suffices to find a comodule on which $\Theta_u$ is not the identity. Projective-space coefficients give the following families of finite examples.

\Needspace{6\baselineskip}
\begin{mainthm}[Global and local algebraic ghosts]\label{thm:algebraic-intro}
At every prime $p$, the following statements hold.
\begin{enumerate}[label=\textup{(\arabic*)},leftmargin=2em]
\item For every $r,n\geq1$, there are a comodule $M\in\D_{\mathrm{fin}}(BP_*BP)$ and a ghost $f:M\to M$ with
\[
 p^r\cdot\id_M=0,\qquad p^{r-1}\cdot f^j\ne0\quad(1\leq j\leq n).
\]
\item For every $h,r,n\geq1$, there are a compact object $M\in L_{K(h)}\operatorname{Stable}(BP_*BP)$ and a ghost $f:M\to M$ with
\[
 p^r\cdot\id_M=0,\qquad p^{r-1}\cdot f^j\ne0\quad(1\leq j\leq n).
\]
\end{enumerate}
In both families, each of $f,f^2,\ldots,f^n$ has exact additive order $p^r$. Thus both parts of Conjecture~\ref{conj:algebraic-gh} are false, and the ghost ideal in each category is not nilpotent.
\end{mainthm}

In the special case $r=n=1$ of \textup{(1)}, we may take
\[
 M=(BP_*/p)\{e_0,e_1\},\qquad |e_0|=0,\quad |e_1|=2(p-1),
\]
with coaction $\psi_M(e_0)=1\otimes e_0$ and $\psi_M(e_1)=1\otimes e_1+t_1\otimes e_0$. We use the Araki generators $v_i$ of $BP_*$ and the standard generators $t_i$ of $BP_*BP$, with the conventions recalled in Section~\ref{subsec:strict-operation}. The comodule $M$ is constructed from four shifts of $BP_*$ and hence belongs to $\D_{\mathrm{fin}}(BP_*BP)$. The ghost $f=1-\Theta_{1+p}$ is given by
\[
 f(e_0)=0,\qquad f(e_1)=v_1e_0.
\]
Theorem~\ref{thm:main} proves that $f$ is nonzero, has exact additive order $p$, and satisfies $f^2=0$.

The analog of Conjecture~\ref{conj:algebraic-gh}(1) for $E_*E$-comodules also fails for every nonrational $E$ as above. Applying Proposition~\ref{prop:intro-criterion} to the same projective-space coefficients gives, for every $r,n\geq1$, a ghost whose first $n$ powers have exact additive order $p^r$ in
\[
 \D_{\mathrm{fin}}(E_*E) =\Thick\{(\Sigma^{2b}E_*)[a]:a,b\in\Z\}.
\]
The general cases are given in Sections~\ref{subsec:projective-spaces} and~\ref{subsec:local-algebraic}.

\medskip
\noindent\textbf{The Motivic Deformation.}
For a field $k$, let $\SH(k)_{\cell}$ be the cellular motivic stable homotopy category, with sphere $\one_k$, and write
\[
 \pi_{a,b}X=[\Sigma^{a,b}\one_k,X].
\]
A finite cellular spectrum lies in the thick subcategory generated by the bigraded spheres; these are the compact objects of $\SH(k)_{\cell}$~\cite[Theorem~9.1 and Proposition~9.4]{DI}. A motivic ghost is a map inducing zero on all these groups. Write $\SH(k)_{\cell,p}^{\wedge}$ for the $p$-complete category, whose unit is $\widehat\one_k=L_{\one_k/p}\one_k$; here $L_{\one_k/p}$ denotes Bousfield localization at $\one_k/p$.

Over $\C$, the comparison of Gheorghe--Wang--Xu~\cite{GWX} places this $BP_*BP$-comodule construction in the special fiber of the motivic deformation of classical stable homotopy theory. Form the cofiber
\[
 C\tau=\cofib(\tau:\Sigma^{0,-1}\widehat\one_{\C} \to\widehat\one_{\C}).
\]
The object $C\tau$ is a commutative algebra~\cite{Gheorghe}. Inverting $\tau$ gives the classical $p$-complete generic fiber~\cite{DIAdams}, while extension of scalars to $C\tau$ gives the special fiber \cite[Section~6, especially Proposition~6.2]{GWX}:
\[
\begin{tikzcd}[column sep=6em]
 \SH_p^{\wedge}
 & \SH(\C)_{\cell,p}^{\wedge}
   \arrow[l,"\tau^{-1}","\textup{generic fiber}"']
   \arrow[r,"\textup{special fiber}","C\tau\wedge_{\widehat\one_{\C}}(-)"']
 & \Mod_{C\tau}\bigl(\SH(\C)_{\cell,p}^{\wedge}\bigr)
   \simeq \operatorname{Stable}(BP_*BP).
\end{tikzcd}
\]
Here $\SH_p^{\wedge}$ is the classical $p$-complete stable homotopy category. In the special-fiber comparison, $\operatorname{Stable}(BP_*BP)$ denotes the $p$-complete even category of comodules over the degreewise $p$-completed Hopf algebroid $BP_*BP$~\cite[Corollary~1.2]{GWX}. The algebraic comparison identifies the bounded derived category of even, degreewise $p$-completed $BP_*BP$-comodules with a full subcategory of $C\tau$-modules, sending the completed comodule unit to $C\tau$~\cite[Theorem~1.1]{GWX}.

\begin{mainconj}[Cellular Motivic Generating Hypothesis]\label{conj:motivic-gh}
Let $k$ be a field. If a map $f:X\to Y$ between finite cellular spectra in $\SH(k)_{\cell}$ satisfies $\pi_{a,b}(f)=0$ for every $(a,b)\in\Z^2$, then $f$ is null-homotopic.
\end{mainconj}

The detectors are the bigraded spheres over $\operatorname{Spec}k$; the condition does not require vanishing on homotopy sheaves evaluated on every smooth $k$-scheme.

The Chow $t$-structure of Bachmann--Kong--Wang--Xu connects $BP_*BP$-comodules to motivic homotopy theory over a general field $k$. For a prime $p$ invertible in $k$, its $p$-local cellular heart in $\SH(k)$ is equivalent to the category of even graded $BP_*BP$-comodules \cite[Theorem~1.9(1) and Section~5.1]{BKWX}. At the level of stable categories, they identify cellular modules over the $p$-local Chow-degree-zero sphere with the even graded part of $\operatorname{Stable}(BP_*BP)$ \cite[Theorem~1.9(2) and Section~5.1]{BKWX}. This equivalence transports the algebraic ghosts to this motivic module category over $k$; see Section~\ref{subsec:fields}.

\Needspace{12\baselineskip}
\subsection{Three types of motivic counterexamples}\label{subsec:counterexamples}

Over $\C$, the motivic deformation exhibits three distinct types of failure of the generating hypothesis. The algebraic ghosts in Theorem~\ref{thm:algebraic-intro}(1) give $C\tau$-linear ghosts in the special fiber. The ambient motivic category contains two further types: non-$C\tau$-linear ghosts between $C\tau$-modules, and ghosts on spectra that admit no $C\tau$-module structure and have noncontractible Betti realization.

\begin{mainthm}[Intrinsic and ambient motivic ghosts]\label{thm:failure}
Conjecture~\ref{conj:motivic-gh} fails in the uncompleted category $\SH(\C)_{\cell}$ in the following three ways.
\begin{enumerate}[label=\textup{(\arabic*)},leftmargin=2em]
\item \textup{($C\tau$-linear ghosts.)} At every prime $p$ and for every $r,n\geq1$, there are a finite cellular $C\tau$-module $X$, whose underlying motivic spectrum is also finite cellular over the uncompleted sphere, and a $C\tau$-linear ghost $\alpha:X\to X$ with
\[
 p^r\cdot\id_X=0,\qquad p^{r-1}\cdot \alpha^j\ne0\quad(1\leq j\leq n).
\]
Thus the first $n$ composition powers of $\alpha$ all have exact additive order $p^r$. When $r=n=1$, $X$ may be chosen with eight cells over the uncompleted sphere.
\item \textup{(Non-$C\tau$-linear ghosts between $C\tau$-modules.)} At every prime $p$ and for every $r\geq1$, there is a $C\tau$-module $Y$ whose underlying motivic spectrum has four cells over the uncompleted sphere and which supports a ghost $\beta:Y\to\Sigma^{1,-1}Y$ of exact additive order $p^r$. For the natural module structures,
\[
 [Y,\Sigma^{1,-1}Y]_{C\tau}=0.
\]
In particular, $\beta$ is not $C\tau$-linear.
\item \textup{(Ghosts on spectra admitting no $C\tau$-module structure.)} At every prime $p$, the four-cell spectrum $Z=\one_{\C}/(p^4,p\tau)$ over the uncompleted sphere supports a nonzero ghost $\gamma:Z\to\Sigma^{1,-1}Z$ with
\[
 \operatorname{Re}(Z)\simeq(S/p^4)\wedge(S/p)\not\simeq *.
\]
Its source and target thus have noncontractible Betti realization, while $\operatorname{Re}(\gamma)=0$. The spectrum $Z$ admits no $C\tau$-module structure.
\end{enumerate}
\end{mainthm}

\noindent\textbf{The intrinsic construction.}
Realizing the comodule $M$ and its ghost $f$ described after Theorem~\ref{thm:algebraic-intro} gives the case $r=n=1$ of Theorem~\ref{thm:failure}(1):
\[
 X=C\tau/(p,\alpha_1),\qquad \alpha=iv_1q:X\longrightarrow X.
\]
Here $i$ and $q$ are the inclusion and quotient in the $\alpha_1$-cofiber sequence. The class $\alpha_1\in\pi_{2p-3,p-1}C\tau$ is represented by the cobar cocycle $[t_1]$ and equals $\eta$ at $p=2$. Figure~\ref{fig:R-cells} displays the four $C\tau$-cells.

\begin{figure}[htbp]
\centering
\begin{tikzpicture}[x=1cm,y=.85cm]
 \foreach \x/\p in {0/s,6/t}{
   \node[cell,label={[labeltext]left:$(0,0)$}] (\p0) at (\x,0) {};
   \node[cell,label={[labeltext]left:$(1,0)$}] (\p1) at (\x,1) {};
   \node[cell,label={[labeltext]left:$(2p-2,p-1)$}] (\p2) at (\x,3) {};
   \node[cell,label={[labeltext]left:$(2p-1,p-1)$}] (\p3) at (\x,4) {};
   \begin{scope}[on background layer]
   \draw (\p0)--node[labeltext,right]{$p$}(\p1);
   \draw (\p2)--node[labeltext,right]{$p$}(\p3);
   \draw (\p0.east) .. controls +(1.1,0) and +(1.1,0) ..
      node[labeltext,right,pos=.4]{$\alpha_1$}(\p2.east);
   \draw (\p1.east) .. controls +(1.7,0) and +(1.7,0) ..
      node[labeltext,right,pos=.6]{$\alpha_1$}(\p3.east);
   \end{scope}
 }
 \draw[rounded corners=3pt] (-2.3,2.65) rectangle (.4,4.35);
 \draw[rounded corners=3pt] (4.6,-.35) rectangle (6.4,1.35);
 \draw[map,preaction={draw=white,line width=3pt}]
   (.4,3.5) to[out=0,in=160]
   node[labeltext,above,pos=.57]{$v_1$}(4.6,.5);
 \node at (0,-.7) {$X$}; \node at (6,-.7) {$X$};
\end{tikzpicture}
\caption{The intrinsic ghost $\alpha=iv_1q$. Each object has four $C\tau$-cells, shown by the vertices, and a model with eight cells over the uncompleted motivic sphere. The arrow $v_1:\Sigma^{2p-2,p-1}(C\tau/p)\to C\tau/p$ joins the framed Moore quotient and subobject. The diagram is drawn in the style of Wang--Xu~\cite[Appendix~I]{WangXu61}.}
\label{fig:R-cells}
\end{figure}
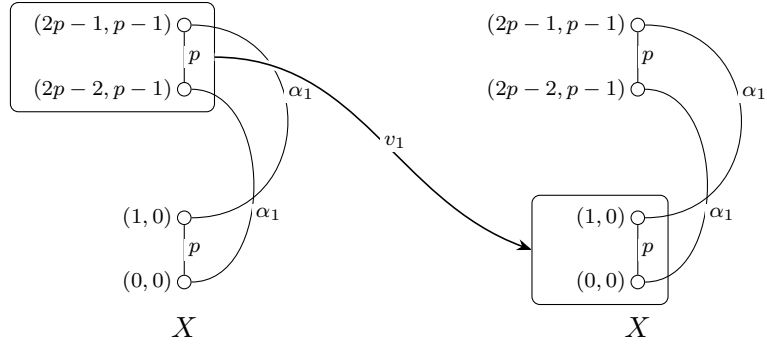

The same comparison realizes the general projective-space families. For every $C\tau$-module $T$, free--forgetful adjunction gives
\[
 [\Sigma^{a,b}C\tau,T]_{C\tau} \cong[\Sigma^{a,b}\widehat\one_{\C},T]=\pi_{a,b}T.
\]
Thus the realized maps remain ghosts after forgetting the module structure. Algebraic cobordism homology $\mathrm{MGL}$ detects the nonzero multiples of their composition powers. Their finite coefficient filtrations give finite cellular models over the uncompleted sphere; see Section~\ref{sec:motivic}.

\noindent\textbf{The ambient Bockstein.}
Part~\textup{(2)} uses the $\tau$-Bockstein
\[
 \widetilde\beta_p=(\Sigma^{1,-1}i)q: C\tau\longrightarrow\Sigma^{1,-1}C\tau,
\]
where $i$ and $q$ belong to the $\tau$-cofiber sequence. At odd primes this map changes weight by one, while the homotopy groups of $C\tau$ occur only in weights divisible by $p-1$. This follows from Isaksen's identification with the classical Adams--Novikov $E_2$-page~\cite[Proposition~6.2.5]{Isaksen}, in its all-prime form~\cite[Corollary~4.8 and Proposition~1.10]{GWX}. The same restriction holds after smashing with $\one_{\C}/p^r$, so the induced map $\beta$ on $Y=\one_{\C}/(p^r,\tau)$ is a ghost. It has exact additive order $p^r$ (Proposition~\ref{prop:complex-quotient} and Figure~\ref{fig:complex-cells}). At $p=2$, take $Y=\one_{\C}/(2^{r+1},\tau)$ and let $\beta$ be twice its Bockstein. Then $\beta$ is a ghost of exact additive order $2^r$ (Propositions~\ref{prop:two-primary-complete} and~\ref{prop:two-primary-finite}). In algebraic cobordism theory $\mathrm{MGL}$, these Bocksteins induce zero on homology, whereas the intrinsic ghosts act nontrivially (Remark~\ref{rem:cobordism-contrast}).

\smallskip
\noindent\textbf{The example with nonzero Betti realization.}
For part~\textup{(3)}, replace the $\tau$-attachment on $\one_{\C}/p^4$ by $p\tau$, giving $Z=\one_{\C}/(p^4,p\tau)$. Let $\beta_Z$ be the Bockstein of this cofiber sequence. Comparison with the $\tau$-cofiber sequence expresses $\gamma=p^3\cdot\beta_Z$ as a composite through a finite Bockstein ghost (Corollary~\ref{cor:p-tau-finite} and Figure~\ref{fig:p-tau-cells}). The spectrum $Z$ has noncontractible Betti realization and no nonzero retract on which a power of $\tau$ acts trivially.

The specialized map $C\tau\wedge\gamma$ is not a ghost (Remark~\ref{rem:p-tau-specialization}). Thus forgetting a module structure preserves the ghost condition, but extension of scalars to $C\tau$ need not do so. All the ghost maps in Theorem~\ref{thm:failure} become null under Betti realization.

\smallskip
\noindent\textbf{Homotopy modules and nonfullness.}
Using the projective-space families, we also show that homotopy modules do not determine finite objects up to equivalence.

\begin{theorem}\label{thm:fullness-classification}
At every prime $p$, there are finite objects $A\not\simeq B$ with isomorphic bigraded homotopy modules in each of the following settings:
\begin{enumerate}[label=\textup{(\arabic*)},leftmargin=2em]
\item $\D_{\mathrm{fin}}(BP_*BP)$, with modules over $\Ext^{*,*}_{BP_*BP}(BP_*,BP_*)$;
\item $\SH(\C)_{\cell,p}^{\wedge}$, with modules over $\pi_{*,*}\widehat\one_{\C}$;
\item $\SH(\C)_{\cell}$ over the uncompleted sphere, with modules over $\pi_{*,*}\one_{\C}$.
\end{enumerate}
\end{theorem}

Theorem~\ref{thm:fullness-classification} has the following consequence for nonfullness.

\begin{corollary}\label{cor:homotopy-not-full}
The functors $\Ext^{*,*}_{BP_*BP}(BP_*,-)$ and $\pi_{*,*}$ in Theorem~\ref{thm:fullness-classification} are not full on finite objects.
\end{corollary}

\smallskip
\noindent\textbf{Homological consequences.}
Taking $r=1$ in Theorems~\ref{thm:algebraic-intro}(1) and \ref{thm:failure}(1) gives arbitrarily long nonzero composites of ghosts on finite objects annihilated by $p$. The ghost-dimension inequalities therefore give detecting modules of arbitrarily large injective and flat dimension over the bigraded endomorphism ring of the unit. All dimensions below are taken in the category of bigraded modules.

\begin{corollary}\label{cor:homological-consequences}
For every prime $p$ and integer $n\geq1$, there are an object $M\in\D_{\mathrm{fin}}(BP_*BP)$ and a finite cellular complex-motivic spectrum $X$, both annihilated by $p$, such that
\[
 \Ext^{*,*}_{BP_*BP}(BP_*,M) \qquad\text{and}\qquad \pi_{*,*}X
\]
each have flat, projective, and injective dimension at least $n$ over the bigraded rings
\[
 \Ext^{*,*}_{BP_*BP}(BP_*,BP_*) \qquad\text{and}\qquad \pi_{*,*}\widehat\one_{\C},
\]
respectively. Consequently, both rings have infinite graded weak global dimension and infinite graded global dimension.
\end{corollary}

\smallskip
\noindent\textbf{Organization of the paper.}
Sections~\ref{sec:complete}--\ref{sec:complex} construct the ambient Bockstein ghosts and their finite models over the motivic sphere. Section~\ref{sec:proof} proves the coefficient criterion and constructs the global and local algebraic ghosts. Section~\ref{sec:motivic} realizes the global algebraic examples as intrinsic motivic ghosts. Section~\ref{sec:fullness-classification} proves Theorem~\ref{thm:fullness-classification} and Corollary~\ref{cor:homotopy-not-full}. Section~\ref{sec:homological-consequences} proves Corollary~\ref{cor:homological-consequences}. Section~\ref{sec:fields} extends the intrinsic examples to real-motivic spectra, including the projective-space families, to algebraically closed fields of characteristic zero, and to Chow-degree-zero module categories over arbitrary fields.

\medskip
\noindent\textbf{Acknowledgments.}
The second author is partially supported by the AMS Centennial Research Fellowship and NSF Grant DMS~2506247.
\FloatBarrier

\Needspace{8\baselineskip}
\section{The \texorpdfstring{$\tau$}{tau}-Bockstein}\label{sec:complete}

Fix a prime $p$. We write $L_p=L_{\one_k/p}$ for $p$-completion, the Bousfield localization with respect to $\one_k/p$, and $\SH(k)_{\cell,p}^{\wedge}$ for the complete category. Its unit is $\widehat\one_k=L_p\one_k$, and its smash product is completed. For a complete spectrum $T$, localization adjunction gives
\[
 [\Sigma^{a,b}\widehat\one_k,T]=[\Sigma^{a,b}\one_k,T]=\pi_{a,b}T.
\]
By Hu--Kriz--Ormsby~\cite[Theorem~1]{HKO}, $\mathrm{H}\F_p$-nilpotent completion and $p$-completion agree on finite cellular spectra over $\C$ at every prime. The main argument is at odd primes; the prime two is treated in Proposition~\ref{prop:two-primary-complete}.

Choose the usual class $\tau:\Sigma^{0,-1}\widehat\one_{\C}\to \widehat\one_{\C}$ and form
\begin{equation}\label{eq:tau-triangle}
 \Sigma^{0,-1}\widehat\one_{\C}\xrightarrow{\tau}\widehat\one_{\C} \xrightarrow{i}C\tau\xrightarrow{q}\Sigma^{1,-1}\widehat\one_{\C}.
\end{equation}
\begin{definition}\label{def:tau-bockstein}
The \emph{$\tau$-Bockstein} is the composite
\begin{equation}\label{eq:beta}
 \widetilde\beta_p: C\tau\xrightarrow{q}\Sigma^{1,-1}\widehat\one_{\C} \xrightarrow{\Sigma^{1,-1}i}\Sigma^{1,-1}C\tau.
\end{equation}
\end{definition}
\begin{proposition}\label{prop:ambient}
Let $p$ be odd. The $\tau$-Bockstein satisfies
\[
 [C\tau,\Sigma^{1,-1}C\tau]\cong \Z_p\{\widetilde\beta_p\}.
\]
The map $\widetilde\beta_p$ is a nonzero ghost of infinite additive order. With the natural $C\tau$-module structures, no nonzero multiple of $\widetilde\beta_p$ is $C\tau$-linear.
\end{proposition}

\begin{proof}
Recall that $BP_*=\Z_{(p)}[v_1,v_2,\ldots]$ and $BP_*BP=BP_*[t_1,t_2,\ldots]$. The generators $v_i,t_i$ have degree $2(p^i-1)$~\cite[Theorems~4.1.18--4.1.19]{Ravenel}. We write $s,t$ for cohomological and internal degree, and call $t/2$ the weight. For any prime $p$ and $B=\Z_{(p)},\Z_p$, or $\Z/p^r$ with $r\geq1$, we have \cite[Note~1.12 and Lemma~1.16]{MRW}
\begin{equation}\label{eq:ext-bound}
 \begin{gathered}
 \Ext^{s,t}_{BP_*BP}(BP_*,BP_*\otimes B)=0\\
 \text{if }s<0,\quad t<2(p-1)s,\quad\text{or }
 t\not\equiv0\pmod{2(p-1)}.
 \end{gathered}
\end{equation}
Tensor products with $B$ and $\Z_p$ are over $\Z_{(p)}$. The normalized cobar complex for $BP_*$ is degreewise finite free, so Ext commutes with flat extension to $\Z_p$ and reduction modulo $p^r$ is unchanged by completion. In internal degree zero, only $\Ext^{0,0}=B$ is nonzero.

The identification of Isaksen~\cite[Proposition~6.2.5]{Isaksen} and Gheorghe~\cite[Corollary~3.14]{Gheorghe} at $p=2$, extended to every prime by Gheorghe--Wang--Xu \cite[Corollary~4.8 and Proposition~1.10]{GWX}, gives
\begin{equation}\label{eq:isaksen}
 \pi_{a,b}C\tau\cong \Ext^{2b-a,2b}_{BP_*BP}(BP_*,BP_*)\otimes\Z_p.
\end{equation}
Consequently, $\pi_{a,b}C\tau=0$ unless $0\leq b\leq a\leq2b$; it also vanishes unless $b\equiv0\pmod{p-1}$. The unit represents $1\in\pi_{0,0}C\tau=\Z_p$.

We use the vanishing region to compute the mapping group, and the restriction on weights to prove the ghost condition. Applying $[-,\Sigma^{1,-1}C\tau]$ to \eqref{eq:tau-triangle} gives
\[
 \pi_{0,1}C\tau\longrightarrow\pi_{0,0}C\tau \xrightarrow{q^*}[C\tau,\Sigma^{1,-1}C\tau] \longrightarrow\pi_{-1,1}C\tau.
\]
The outer groups vanish, and $q^*$ sends the shifted unit to $\widetilde\beta_p$. Hence
\begin{equation}\label{eq:complete-maps}
 [C\tau,\Sigma^{1,-1}C\tau]\cong\Z_p\{\widetilde\beta_p\}.
\end{equation}
The homomorphism
\[
 (\widetilde\beta_p)_*:\pi_{a,b}C\tau \longrightarrow\pi_{a-1,b+1}C\tau
\]
is zero: its source and target weights differ by one, whereas nonzero groups have weights divisible by $p-1$. Finally,
\[
 [C\tau,\Sigma^{1,-1}C\tau]_{C\tau} =\pi_{-1,1}C\tau=0,
\]
which proves that no nonzero multiple of the Bockstein is $C\tau$-linear.
\end{proof}

The same restriction on weights also controls $\tau$-torsion in the completed sphere.

\begin{corollary}\label{cor:tau-divisibility}
Let $p$ be odd. If $x\in\pi_{a,b}\widehat\one_{\C}$ satisfies $\tau\cdot x=0$, then $x$ is $\tau^{p-2}$-divisible.
\end{corollary}
\begin{proof}
We may assume $x\ne0$. Choose $y\in\pi_{a+1,b-1}C\tau$ with $q_*(y)=x$. Then $y\ne0$, so $b\equiv1\pmod{p-1}$. For $0\leq j\leq p-3$, the group $\pi_{a,b+j}C\tau$ vanishes, so exactness makes
\[
 \tau:\pi_{a,b+j+1}\widehat\one_{\C} \longrightarrow\pi_{a,b+j}\widehat\one_{\C}
\]
surjective. Lifting $x$ successively gives $x=\tau^{p-2}\cdot z$.
\end{proof}

Corollary~\ref{cor:tau-divisibility} also follows from the rigidity of the motivic Adams--Novikov spectral sequence at odd primes \cite[Propositions~5.3 and~5.6]{Stahn}.

For $p=2$, we need the following lemma, which is a special case of Burklund's torsion theorem~\cite[Corollary~2.6]{Burklund}. For completeness, we include a proof.

\begin{lemma}\label{lem:odd-weight-torsion}
At $p=2$, we have
\begin{equation}\label{eq:coefficient-torsion}
 2\cdot\Ext^{s,4k+2}_{BP_*BP}(BP_*,BP_*)=0 \qquad(s\geq0,\ k\in\Z).
\end{equation}
\end{lemma}
\begin{proof}
Let $C^*_{BP_*BP}(BP_*)$ be Ravenel's cobar complex~\cite[A1.2.11--A1.2.12]{Ravenel}. Let $\sigma(a)=(-1)^{|a|/2}a$ for homogeneous $a$ in $BP_*$ or $BP_*BP$. In the formal-group description~\cite[Theorem~4.1.19]{Ravenel}, Novikov's Adams operation $\Psi_U^{-1}$ \cite[Appendix~2, Theorem~2(a)(2)--(4)]{Novikov} gives the strict isomorphism
\[
 h_F(x)=-[-1]_F(x):F\xrightarrow{\cong}\sigma F, \qquad (\sigma F)(x,y)=-F(-x,-y),
\]
for every $2$-typical formal group law $F$, where $[-1]_F$ is its formal inverse. This is natural: for every strict isomorphism $f:F\to G$,
\[
 h_G\circ f=(\sigma f)\circ h_F, \qquad (\sigma f)(x)=-f(-x).
\]
Thus $\sigma$ and the identity are naturally equivalent Hopf algebroid maps, so they induce the same map on $\Ext_{BP_*BP}(BP_*,BP_*)$ \cite[Definition~6.1.4 and proof of Lemma~6.1.5]{Ravenel}. But $\sigma$ acts on internal degree $t$ of $C^*_{BP_*BP}(BP_*)$ by $(-1)^{t/2}$. Hence every $x\in\Ext^{s,4k+2}_{BP_*BP}(BP_*,BP_*)$ satisfies $x=\sigma_*(x)=-x$, as required.
\end{proof}

\begin{proposition}\label{prop:two-primary-complete}
Let $p=2$. The $\tau$-Bockstein satisfies
\[
 [C\tau,\Sigma^{1,-1}C\tau]\cong \Z_2\{\widetilde\beta_2\}.
\]
The map $2\cdot\widetilde\beta_2$ is a nonzero ghost of infinite additive order. With the natural $C\tau$-module structures, no nonzero multiple of $\widetilde\beta_2$ is $C\tau$-linear.
\end{proposition}
\begin{proof}
At $p=2$, the Bockstein is studied in \cite[\S4.2]{Gheorghe}, and \eqref{eq:isaksen} is Isaksen's theorem \cite[Proposition~6.2.5]{Isaksen}. The mapping-group and $C\tau$-linearity calculations in the proof of Proposition~\ref{prop:ambient} apply unchanged. In particular, $2\cdot\widetilde\beta_2$ is nonzero and has infinite additive order.

For the ghost condition, Lemma~\ref{lem:odd-weight-torsion} and \eqref{eq:isaksen} give $2\cdot\pi_{a,b}C\tau=0$ when $b$ is odd. Exactly one of $b$ and $b+1$ is odd, so $2\cdot\widetilde\beta_2$ induces zero on every homotopy group.
\end{proof}

The following is the prime-two analog of Corollary~\ref{cor:tau-divisibility}.

\begin{corollary}\label{cor:two-primary-tau-divisibility}
Let $p=2$. If $x\in\pi_{a,b}\widehat\one_{\C}$ satisfies $\tau\cdot x=0$, then $2\cdot x$ is $\tau$-divisible.
\end{corollary}
\begin{proof}
Choose $y\in\pi_{a+1,b-1}C\tau$ with $q_*(y)=x$. Proposition~\ref{prop:two-primary-complete} gives
\[
 i_*(2\cdot x)=(2\cdot\widetilde\beta_2)_*(y)=0.
\]
Exactness shows that $2\cdot x$ is $\tau$-divisible.
\end{proof}

Propositions~\ref{prop:ambient} and~\ref{prop:two-primary-complete} use finitely many cells over the completed sphere. We next construct finite cellular examples in the uncompleted category.

\section{Four-cell integral counterexamples}\label{sec:complex}

Fix a prime $p$ and an integer $r\geq1$. We prove Theorem~\ref{thm:failure}\textup{(2)} and \textup{(3)} using $\tau$- and $p\tau$-Bocksteins on four-cell spectra over the uncompleted sphere.

\begin{lemma}\label{lem:descent}
Over any field $k$, objects built from $\one_k/p^r$ by bigraded shifts, finite cofibers, and retracts are $p$-complete. For finite cellular $T$, the ordinary smash product satisfies
\[
 \widehat\one_k\wedge T\simeq L_pT.
\]
\end{lemma}
\begin{proof}
The Moore sequence and orthogonality to spectra on which $p$ is invertible show that $\one_k/p^r$ and every spectrum killed by a power of $p$ are complete. Closure under shifts, finite cofibers, and retracts proves the first assertion and completeness of $\widehat\one_k\wedge T$ for finite cellular $T$. The natural mod-$p$ equivalence $T\to\widehat\one_k\wedge T$ is therefore $p$-completion.
\end{proof}

\begin{definition}\label{def:finite-tau-bockstein}
Define $Y=\one_{\C}/(p^r,\tau)$ by the cofiber sequence
\begin{equation}\label{eq:complex-integral-triangle}
 \Sigma^{0,-1}(\one_{\C}/p^r)\xrightarrow{\tau}\one_{\C}/p^r \xrightarrow{i}Y\xrightarrow{q}\Sigma^{1,-1}(\one_{\C}/p^r).
\end{equation}
Its \emph{$\tau$-Bockstein} is the composite
\begin{equation}\label{eq:complex-integral-map}
 \beta_{p^r}: Y\xrightarrow{q}\Sigma^{1,-1}(\one_{\C}/p^r) \xrightarrow{\Sigma^{1,-1}i}\Sigma^{1,-1}Y.
\end{equation}
\end{definition}
By Lemma~\ref{lem:descent}, smashing the completed $\tau$-map with $\one_{\C}/p^r$ induces $\tau$ above. This gives $Y\simeq C\tau\wedge(\one_{\C}/p^r)\simeq C\tau/p^r$ and $\beta_{p^r}=\widetilde\beta_p\wedge\id_{\one_{\C}/p^r}$. Thus $Y$ is a $p$-complete $C\tau$-module with four cells over the uncompleted sphere, in bidegrees $(0,0),(1,0),(1,-1),(2,-1)$.

\begin{proposition}\label{prop:complex-quotient}
At every prime $p$, with the natural $C\tau$-module structures, we have
\begin{equation}\label{eq:complex-quotient-maps}
 [Y,\Sigma^{1,-1}Y]\cong\Z/p^r\{\beta_{p^r}\},\qquad [Y,\Sigma^{1,-1}Y]_{C\tau}=0.
\end{equation}
For odd $p$, the map $\beta_{p^r}$ is a non-$C\tau$-linear ghost of exact additive order $p^r$ on the four-cell spectrum $Y$ over the uncompleted sphere.
\end{proposition}
\begin{proof}
The Moore cofiber sequence $C\tau\xrightarrow{p^r}C\tau\xrightarrow{j}Y$ gives
\begin{equation}\label{eq:moore-exact}
 0\longrightarrow(\pi_{a,b}C\tau)/p^r \longrightarrow\pi_{a,b}Y \longrightarrow(\pi_{a-1,b}C\tau)[p^r]\longrightarrow0,
\end{equation}
where $G[p^r]=\ker(p^r:G\to G)$. At odd primes, the weight support of $C\tau$ implies
\begin{equation}\label{eq:quotient-sparseness}
 \pi_{a,b}Y=0\qquad\text{unless }b\equiv0\pmod{p-1}.
\end{equation}
Since $\beta_{p^r}$ changes weight by one, it is a ghost.

For the mapping group, \eqref{eq:ext-bound} and \eqref{eq:moore-exact}, together with $\pi_{0,0}C\tau=\Z_p$, give at every prime
\begin{equation}\label{eq:complex-low}
 \pi_{0,0}Y=\Z/p^r,\qquad \pi_{1,0}Y=\pi_{0,1}Y=\pi_{-1,1}Y=0.
\end{equation}
The $\tau$- and Moore cofiber sequences then give
\[
 [Y,\Sigma^{1,-1}Y]\xrightarrow[\cong]{j^*} [C\tau,\Sigma^{1,-1}Y]\cong\Z/p^r.
\]
Under these isomorphisms, $\beta_{p^r}$ corresponds to $1$, since $\beta_{p^r}j=(\Sigma^{1,-1}j)\widetilde\beta_p$. Lemma~\ref{lem:descent} identifies these with maps in the uncompleted category.

For the module mapping group, applying $[-,\Sigma^{1,-1}Y]_{C\tau}$ to $C\tau\xrightarrow{p^r}C\tau\to Y$ gives outer groups $\pi_{0,1}Y$ and $\pi_{-1,1}Y$. Both vanish by \eqref{eq:complex-low}, proving $[Y,\Sigma^{1,-1}Y]_{C\tau}=0$.
\end{proof}

Figure~\ref{fig:complex-cells} depicts the map $\beta_{p^r}$.

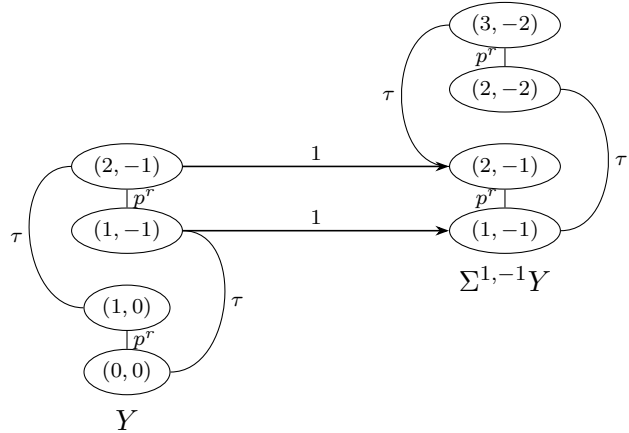
\begin{figure}[htbp]
\centering
\begin{tikzpicture}[x=1cm,y=.85cm]
 \node[motivic cell] (c00) at (0,0) {$(0,0)$};
 \node[motivic cell] (c10) at (0,1) {$(1,0)$};
 \node[motivic cell] (c01) at (0,2.2) {$(1,-1)$};
 \node[motivic cell] (c11) at (0,3.2) {$(2,-1)$};
 \node[motivic cell] (d00) at (5,2.2) {$(1,-1)$};
 \node[motivic cell] (d10) at (5,3.2) {$(2,-1)$};
 \node[motivic cell] (d01) at (5,4.4) {$(2,-2)$};
 \node[motivic cell] (d11) at (5,5.4) {$(3,-2)$};
 \draw (c00)--node[cell label,right] {$p^r$}(c10);
 \draw (c01)--node[cell label,right] {$p^r$}(c11);
 \draw (d00)--node[cell label,left] {$p^r$}(d10);
 \draw (d01)--node[cell label,left] {$p^r$}(d11);
 \draw (c11.west) to[out=180,in=180,looseness=1.15]
   node[cell label,left] {$\tau$} (c10.west);
 \draw (c01.east) to[out=0,in=0,looseness=1.15]
   node[cell label,right] {$\tau$} (c00.east);
 \draw (d11.west) to[out=180,in=180,looseness=1.15]
   node[cell label,left] {$\tau$} (d10.west);
 \draw (d01.east) to[out=0,in=0,looseness=1.15]
   node[cell label,right] {$\tau$} (d00.east);
 \draw[cell map] (c01)--node[cell label,above] {$1$}(d00);
 \draw[cell map] (c11)--node[cell label,above] {$1$}(d10);
 \node at (0,-.75) {$Y$};
 \node at (5,1.45) {$\Sigma^{1,-1}Y$};
\end{tikzpicture}
\caption{The odd-primary ghost $\beta_{p^r}$ of Proposition~\ref{prop:complex-quotient}, where $p$ is odd and $r\geq1$. Here $Y=\one_{\C}/(p^r,\tau)$ has four cells over the uncompleted motivic sphere. The arrows identify the top Moore quotient with the shifted bottom Moore subcomplex.}
\label{fig:complex-cells}
\end{figure}

\Needspace{9\baselineskip}
\begin{proposition}[The prime two]\label{prop:two-primary-finite}
Suppose $p=2$ and $r\geq2$. The map $2\cdot\beta_{2^r}:Y\to\Sigma^{1,-1}Y$ is a nonzero ghost of exact additive order $2^{r-1}$. For the natural $C\tau$-module structures, it is not $C\tau$-linear.
\end{proposition}
\begin{proof}
For any prime $p$ and $r\geq1$, the exact comparison of Gheorghe--Wang--Xu~\cite[Theorem~1.1]{GWX} sends $C\tau/p^r$ to $BP_*/p^r$, since multiplication by $p^r$ is injective on the completed comodule unit $BP_*\otimes\Z_p$. Free-module adjunction and the grading in \eqref{eq:isaksen} give the natural isomorphism
\begin{equation}\label{eq:coefficient-comparison}
 \pi_{a,b}(C\tau/p^r)\cong \Ext^{2b-a,2b}_{BP_*BP}(BP_*,BP_*/p^r).
\end{equation}
Here completed and uncompleted coefficient Ext groups agree by the cobar calculation in Section~\ref{sec:complete}. These groups vanish unless $0\leq b\leq a\leq2b$, by \eqref{eq:ext-bound}.

At $p=2$, the integral cobar complex is degreewise finite free over $\Z_{(2)}$, so its universal coefficient sequence splits noncanonically in each internal degree. Lemma~\ref{lem:odd-weight-torsion} gives
\begin{equation}\label{eq:coefficient-bounds}
 2\cdot\Ext^{s,2w}_{BP_*BP}(BP_*,BP_*/2^r)=0 \qquad(w\text{ odd}).
\end{equation}
Thus $2\cdot\pi_{a,b}Y=0$ for odd $b$. Since $\beta_{2^r}$ changes weight by one, $2\cdot\beta_{2^r}$ is a ghost.

For the remaining assertions, Proposition~\ref{prop:complex-quotient} gives $\beta_{2^r}$ exact additive order $2^r$ and $[Y,\Sigma^{1,-1}Y]_{C\tau}=0$. Hence $2\cdot\beta_{2^r}$ has exact additive order $2^{r-1}$ and is not $C\tau$-linear.
\end{proof}

\Needspace{6\baselineskip}
\begin{proof}[Proof of Theorem~\ref{thm:failure}\textup{(2)}]
For odd $p$, take $Y=\one_{\C}/(p^r,\tau)$ and $\beta=\beta_{p^r}$, as in Proposition~\ref{prop:complex-quotient}. At $p=2$, take $Y=\one_{\C}/(2^{r+1},\tau)$ and $\beta=2\cdot\beta_{2^{r+1}}$, using Proposition~\ref{prop:two-primary-finite}. In either case, $Y$ has four cells over the uncompleted sphere and $\beta$ is a ghost of exact additive order $p^r$. For the natural $C\tau$-module structures, $[Y,\Sigma^{1,-1}Y]_{C\tau}=0$ by Proposition~\ref{prop:complex-quotient}.
\end{proof}

\FloatBarrier

\subsection{A ghost on the cofiber of \texorpdfstring{$p\tau$}{p tau}}\label{subsec:p-tau}

Let $p$ be any prime. Replacing the attachment $\tau$ by $p\tau$ gives a cofiber whose completed Betti realization is $S/p$. Comparing the two cofiber sequences will produce a ghost on this spectrum.

\begin{definition}\label{def:completed-p-tau-bockstein}
Define $D=\widehat\one_{\C}/(p\tau)$ by the cofiber sequence
\begin{equation}\label{eq:p-tau-triangle}
 \Sigma^{0,-1}\widehat\one_{\C} \xrightarrow{p\tau}\widehat\one_{\C} \xrightarrow{i_D}D\xrightarrow{q_D} \Sigma^{1,-1}\widehat\one_{\C}.
\end{equation}
Its \emph{$p\tau$-Bockstein} is $\beta_D=(\Sigma^{1,-1}i_D)q_D$.
\end{definition}
The spectrum $D$ has two cells over the completed sphere.

\begin{proposition}\label{prop:p-tau}
There is an isomorphism
\[
 [D,\Sigma^{1,-1}D]\cong\Z_p\{\beta_D\}.
\]
The map $p^3\cdot\beta_D$ is a nonzero ghost of infinite additive order. Moreover,
\[
 \operatorname{Re}_p(D)\simeq S/p\not\simeq *, \qquad \operatorname{Re}_p(p^3\cdot\beta_D)=0.
\]
In particular, $D$ admits no $C\tau$-module structure.
\end{proposition}

\begin{proof}
The two commutative squares
\[
\begin{tikzcd}[column sep=2.4em,row sep=1.8em]
 \Sigma^{0,-1}\widehat\one_{\C}
   \arrow[r,"p\tau"]\arrow[d,"p"']
 &\widehat\one_{\C}\arrow[d,equal]\\
 \Sigma^{0,-1}\widehat\one_{\C}\arrow[r,"\tau"']
 &\widehat\one_{\C}
\end{tikzcd}
\qquad
\begin{tikzcd}[column sep=2.4em,row sep=1.8em]
 \Sigma^{0,-1}\widehat\one_{\C}
   \arrow[r,"\tau"]\arrow[d,equal]
 &\widehat\one_{\C}\arrow[d,"p"]\\
 \Sigma^{0,-1}\widehat\one_{\C}\arrow[r,"p\tau"']
 &\widehat\one_{\C}
\end{tikzcd}
\]
induce maps $c:D\to C\tau$ and $j:C\tau\to D$ satisfying
\[
 ci_D=i,\qquad qc=p\cdot q_D,\qquad ji=p\cdot i_D,\qquad q_Dj=q.
\]
Consequently,
\begin{equation}\label{eq:p-tau-factorization}
 (\Sigma^{1,-1}j)(p\cdot\widetilde\beta_p)c =p^3\cdot\beta_D.
\end{equation}
The middle map is a ghost by Proposition~\ref{prop:ambient} at odd primes and Proposition~\ref{prop:two-primary-complete} at $p=2$. Since ghosts are closed under pre- and postcomposition, $p^3\cdot\beta_D$ is a ghost.

Motivic connectivity and the completion comparison \cite[Theorem~1 and Lemma~11]{HKO} give $\pi_{a,b}\widehat\one_{\C}=0$ for $a<b$. The $\tau$-cofiber sequence therefore identifies $\pi_{0,0}\widehat\one_{\C}$ with $\pi_{0,0}C\tau=\Z_p$. The $p\tau$-cofiber sequence now gives
\[
 \pi_{0,0}D=\Z_p,\qquad \pi_{0,1}D=\pi_{-1,1}D=0.
\]
Applying $[-,\Sigma^{1,-1}D]$ to the $p\tau$-cofiber sequence gives
\[
 [D,\Sigma^{1,-1}D]\cong\pi_{0,0}D=\Z_p, \qquad \beta_D\longleftrightarrow1.
\]
Hence $p^3\cdot\beta_D$ is nonzero and has infinite additive order.

Let $i_p:S\to S/p$ and $q_p:S/p\to\Sigma S$ be the maps in the classical Moore cofiber sequence. Completed Betti realization sends $D$ to $S/p$ and $\beta_D$ to the Moore Bockstein $(\Sigma i_p)q_p$. Since $p\cdot i_p=0$, this Bockstein is annihilated by $p$, so $\operatorname{Re}_p(p^3\cdot\beta_D)=0$. Every $C\tau$-module has contractible completed Betti realization, so $D$ admits no such module structure.
\end{proof}

To obtain a finite spectrum over the uncompleted sphere, we pass from $D$ to $D/p^4$. The choice of $p^4$ leaves the multiple $p^3$ detectable in $\Z/p^4$.

\begin{definition}\label{def:finite-p-tau-bockstein}
Define
\[
 Z=\cofib\bigl(p\tau:\Sigma^{0,-1}(\one_{\C}/p^4) \to\one_{\C}/p^4\bigr) =\one_{\C}/(p^4,p\tau), \qquad \gamma=p^3\cdot\beta_Z,
\]
where $\beta_Z$ is the projection to the top Moore quotient followed by the shifted inclusion of the bottom Moore subcomplex.
\end{definition}
Lemma~\ref{lem:descent} defines this $\tau$-action on $\one_{\C}/p^4$ and identifies $Z$ with $D/p^4$.

\begin{corollary}\label{cor:p-tau-finite}
The map $\gamma:Z\to\Sigma^{1,-1}Z$ is a nonzero ghost on a four-cell motivic spectrum over the uncompleted sphere. Its Betti realization satisfies
\[
 \operatorname{Re}(Z)\simeq(S/p^4)\wedge(S/p)\not\simeq *, \qquad \operatorname{Re}(\gamma)=0.
\]
The spectrum $Z$ has no nonzero retract on which a power of $\tau$ acts trivially. In particular, it admits no $C\tau$-module structure.
\end{corollary}

\begin{figure}[htbp]
\centering
\begin{tikzpicture}[x=1cm,y=.85cm]
 \node[motivic cell] (c00) at (0,0) {$(0,0)$};
 \node[motivic cell] (c10) at (0,1) {$(1,0)$};
 \node[motivic cell] (c01) at (0,2.2) {$(1,-1)$};
 \node[motivic cell] (c11) at (0,3.2) {$(2,-1)$};
 \node[motivic cell] (d00) at (5,2.2) {$(1,-1)$};
 \node[motivic cell] (d10) at (5,3.2) {$(2,-1)$};
 \node[motivic cell] (d01) at (5,4.4) {$(2,-2)$};
 \node[motivic cell] (d11) at (5,5.4) {$(3,-2)$};
 \draw (c00)--node[cell label,right] {$p^4$}(c10);
 \draw (c01)--node[cell label,right] {$p^4$}(c11);
 \draw (d00)--node[cell label,left] {$p^4$}(d10);
 \draw (d01)--node[cell label,left] {$p^4$}(d11);
 \draw (c11.west) to[out=180,in=180,looseness=1.15]
   node[cell label,left] {$p\tau$} (c10.west);
 \draw (c01.east) to[out=0,in=0,looseness=1.15]
   node[cell label,right] {$p\tau$} (c00.east);
 \draw (d11.west) to[out=180,in=180,looseness=1.15]
   node[cell label,left] {$p\tau$} (d10.west);
 \draw (d01.east) to[out=0,in=0,looseness=1.15]
   node[cell label,right] {$p\tau$} (d00.east);
 \draw[cell map] (c01)--node[cell label,above] {$p^3$}(d00);
 \draw[cell map] (c11)--node[cell label,above] {$p^3$}(d10);
 \node at (0,-.75) {$Z$};
 \node at (5,1.45) {$\Sigma^{1,-1}Z$};
\end{tikzpicture}
\caption{The ghost $\gamma=p^3\cdot\beta_Z$ of Corollary~\ref{cor:p-tau-finite} on $Z=\one_{\C}/(p^4,p\tau)$. All cells are suspensions of the uncompleted motivic sphere. The $p\tau$-lines collectively record the attachment on $\one_{\C}/p^4$. The two arrows give multiplication by $p^3$ from the top Moore quotient to the shifted bottom Moore subcomplex.}
\label{fig:p-tau-cells}
\end{figure}
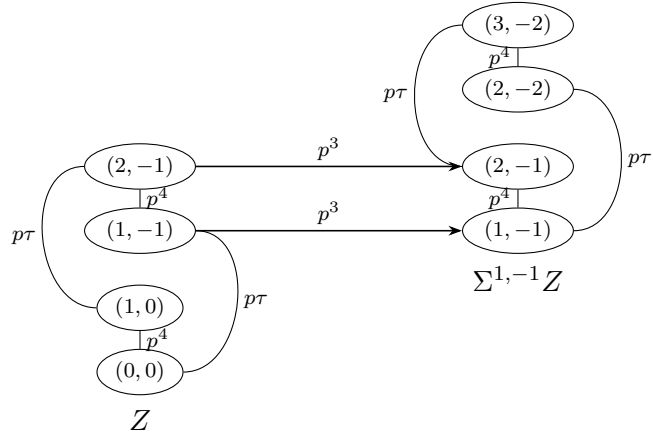

\begin{proof}
The cells of $Z$, displayed in Figure~\ref{fig:p-tau-cells}, have bidegrees $(0,0),(1,0),(1,-1),(2,-1)$, and $Z$ is $p$-complete by Lemma~\ref{lem:descent}. Replacing the completed sphere in the two comparison squares by $\one_{\C}/p^4$ gives maps $c:Z\to C\tau/p^4$ and $j:C\tau/p^4\to Z$, with
\[
 (\Sigma^{1,-1}j)(p\cdot\beta_{p^4})c =p^3\cdot\beta_Z=\gamma.
\]
Propositions~\ref{prop:complex-quotient} and \ref{prop:two-primary-finite} show that the middle map is a ghost at every prime.

To check that reduction modulo $p^4$ has not killed the map, we restrict along the quotient $a:D\to Z$. The same vanishing argument, applied to the Moore and $p\tau$-cofiber sequences, gives $\pi_{0,0}Z=\Z/p^4$ and $\pi_{0,1}Z=\pi_{-1,1}Z=0$. Applying $[-,\Sigma^{1,-1}Z]$ to \eqref{eq:p-tau-triangle} gives
\[
 [D,\Sigma^{1,-1}Z]\cong\Z/p^4, \qquad \beta_Z a\longleftrightarrow1.
\]
Thus $\gamma a$ corresponds to $p^3\ne0$ in $\Z/p^4$.

The Betti realization of $Z$ is built from $S/p^4$ and is therefore already $p$-complete. The attachment $p\tau$ becomes multiplication by $p$ on $S/p^4$, and the Bockstein becomes $\id_{S/p^4}\wedge((\Sigma i_p)q_p)$. Since $p\cdot i_p=0$, these observations prove the realization claims.

For the assertion about retracts, observe that the attachments $p^4$ and $p\tau$ become null after smashing with the mod-$p$ motivic Eilenberg--Mac Lane spectrum $\mathrm{H}\F_p$. Thus the mod-$p$ motivic homology of $Z$ is free over $\F_p[\tau]$. If $K$ is a retract of $Z$ killed by a power of $\tau$, its mod-$p$ motivic homology is both $\tau$-torsion and a direct summand of a $\tau$-torsion-free module, hence zero. Since $Z$ is finite and $p$-complete, the theorem of Hu--Kriz--Ormsby~\cite[Theorem~1]{HKO} implies that $Z$ and its retracts are $\mathrm{H}\F_p$-complete. Thus $K\simeq *$.
\end{proof}

This proves Theorem~\ref{thm:failure}(3) at every prime.

\begin{remark}\label{rem:p-tau-specialization}
Extension of scalars $C\tau\wedge_{\widehat\one_{\C}}(-)$ from $\SH(\C)_{\cell,p}^{\wedge}$ to $C\tau$-modules need not preserve ghosts. In $C\tau$-modules, the $p\tau$-cofiber sequence splits, giving
\[
 C\tau\wedge Z\simeq C\tau/p^4 \vee\Sigma^{1,-1}(C\tau/p^4).
\]
Under a compatible splitting, $C\tau\wedge\gamma$ carries the top summand to the shifted bottom summand by multiplication by $p^3$. The class $1\in\pi_{0,0}(C\tau/p^4)=\Z/p^4$ gives a spherical class in the top summand whose image is $p^3\ne0$. Thus the specialized map is not a ghost. The original ghost factors through $C\tau/p^4$ and becomes null on the generic fiber, although its source has noncontractible generic fiber and no nonzero retract on which a power of $\tau$ acts trivially.
\end{remark}

\FloatBarrier

\section{The algebraic generating hypothesis}\label{sec:proof}

We use formal multiplication to construct algebraic ghosts. After giving a finite coefficient example and computing its endomorphism ring, we prove Theorem~\ref{thm:algebraic-intro}\textup{(1)} using projective-space comodules in Section~\ref{subsec:projective-spaces}. We then extend the construction to local categories in Section~\ref{subsec:local-algebraic} and prove Theorem~\ref{thm:algebraic-intro}\textup{(2)}.

\subsection{Conventions}\label{subsec:categories}

Fix a prime $p$. We use even graded left $BP_*BP$-comodules, with
\[
 |v_i|=|t_i|=2(p^i-1),\qquad (\Sigma^tL)_n=L_{n-t}.
\]
We use cochain indexing for complexes, with $(K[r])^n=K^{n+r}$ and differential $(-1)^r\partial_K$. For a comodule $L$ and an even integer $t$, our convention is
\[
 \Ext^{s,t}_{BP_*BP}(BP_*,L) =\Hom_{\D(BP_*BP)}(\Sigma^tBP_*,L[s]).
\]
Tensor products in the derived category are derived over $BP_*$. Bounded complexes whose terms are flat over $BP_*$ compute them by ordinary tensor products.

We call an object finite if it belongs to $\D_{\mathrm{fin}}(BP_*BP)=\Thick\{(\Sigma^{2b}BP_*)[a]:a,b\in\Z\}$. This subcategory identifies with the compact subcategory of the even graded part of Hovey's stable category, which is obtained by taking the Ind-completion of the thick subcategory of dualizable comodules inside the derived category \cite[Definition~3.1 and Propositions~3.2--3.3]{BH}. The stable-to-derived comparison is fully faithful on bounded objects \cite[Proposition~3.6 and Lemma~3.16]{BH}. Odd internal shifts of the unit have no maps to even objects, so our examples remain ghosts in the full graded stable category. The same argument applies to $E_*E$-comodules and, by localization adjunction, to the local examples.

\subsection{Ghosts from formal multiplication}\label{subsec:strict-operation}

We first construct natural automorphisms of coefficient comodules and show that their differences from the identity are ghosts. Let $E$ be an even, $p$-local Landweber exact theory with a $p$-typical orientation. We use the ordinary flat Hopf algebroid \cite[Introduction and Corollary~2.3]{HS}
\[
 (E_*,E_*E),\qquad E_*E\cong E_*\otimes_{BP_*}BP_*BP\otimes_{BP_*}E_*.
\]
Let $F$ be the formal group law of the chosen orientation. Recall that $BP_*$ classifies $p$-typical formal group laws and $BP_*BP$ classifies strict isomorphisms between them, that is, isomorphisms with linear term $x$ \cite[Theorems~4.1.19 and~A2.1.27]{Ravenel}. By the displayed base change, $E_*E$ classifies strict isomorphisms between pullbacks of $F$ along two coefficient maps from $E_*$. The two unit maps $\eta_L$ and $\eta_R$ record the source and target laws. We use the Araki generators $v_i$ and the standard $t_i$-coordinates~\cite[Lemma~A2.1.26 and equation~(A2.2.2)]{Ravenel}. We also write $v_i$ for their images in $E_*$, and $\Delta$ for the coproduct on $E_*E$. The definitions of internal shifts, Ext groups, and ghosts are the same as for $BP_*BP$. In this subsection, comodules may have arbitrary internal degrees.

Fix $r\geq1$. The ideal $(p^r)$ is invariant, and $(E_*/p^r,E_*E/p^r)$ is a flat quotient Hopf algebroid.
\begin{definition}\label{def:formal-multiplication-operation}
For an integer $u\equiv1\pmod{p^r}$, formal multiplication $[u]_F$ is a strict automorphism of the formal group over $E_*/p^r$. Let
\[
 \theta_u:E_*E/p^r\longrightarrow E_*/p^r
\]
classify this automorphism. For a graded $E_*E/p^r$-comodule $L$ with coaction
\[
 \psi_L:L\longrightarrow (E_*E/p^r)\otimes_{E_*/p^r}L,
\]
define $\Theta_u$ by the composite
\begin{equation}\label{eq:Theta}
 \begin{tikzcd}[column sep=large]
 L \arrow[r,"\psi_L"] &
 (E_*E/p^r)\otimes_{E_*/p^r}L
 \arrow[r,"\theta_u\otimes1"] &
 (E_*/p^r)\otimes_{E_*/p^r}L
 \arrow[r,"\cong"] & L.
 \end{tikzcd}
\end{equation}
\end{definition}

\begin{lemma}\label{lem:principal-unit-automorphism}
For every integer $u\equiv1\pmod{p^r}$, $\Theta_u$ is a natural automorphism of the identity functor on graded $E_*E/p^r$-comodules. It is the identity on every internal shift of $E_*/p^r$. For integers $u,v\equiv1\pmod{p^r}$, we have $\Theta_u\Theta_v=\Theta_{uv}$.
\end{lemma}
\begin{proof}
Since $\theta_u$ classifies $[u]_F:F\to F$ over $E_*/p^r$, composing either unit map with $\theta_u$ gives the identity. Over $E_*E/p^r$, let $h:\eta_L^*F\to\eta_R^*F$ be the universal strict isomorphism. As a formal-group homomorphism, it satisfies $h\circ[u]_{\eta_L^*F}=[u]_{\eta_R^*F}\circ h$. The coproduct classifies composition, so these two composites are classified by $(\theta_u\otimes1)\Delta$ and $(1\otimes\theta_u)\Delta$, respectively. Hence
\begin{equation}\label{eq:Theta-centrality}
 (\theta_u\otimes1)\Delta=(1\otimes\theta_u)\Delta.
\end{equation}
The identity $\theta_u\eta_R=\id$ makes $\theta_u\otimes1$ well-defined on the tensor product in \eqref{eq:Theta}: it sends multiplication by $\eta_R(a)$ on the first factor to multiplication by $a$ on $L$. For $a\in E_*/p^r$ and $\ell\in L$, the coaction satisfies $\psi_L(a\ell)=\eta_L(a)\psi_L(\ell)$. Since $\theta_u\eta_L=\id$, we obtain
\[
 \Theta_u(a\ell) =(\theta_u\otimes1)\bigl(\eta_L(a)\psi_L(\ell)\bigr) =a\Theta_u(\ell).
\]
Thus $\Theta_u$ is $E_*/p^r$-linear. To show that it is a comodule map, we must prove $\psi_L\Theta_u=(1\otimes\Theta_u)\psi_L$. The following diagram expands the maps involved. All tensor products are over $E_*/p^r$, with canonical unit identifications understood.
\[
 \begin{tikzcd}[column sep=4em,row sep=2em]
 L \arrow[r,"\psi_L"] \arrow[dd,"\psi_L"']
 \arrow[rr,bend left=22,"\Theta_u"] &
 (E_*E/p^r)\otimes L \arrow[r,"\theta_u\otimes1"]
 \arrow[d,shift right=1.1ex,"\Delta\otimes1"']
 \arrow[d,shift left=1.1ex,"1\otimes\psi_L"] &
 L \arrow[dd,"\psi_L"] \\
 & (E_*E/p^r)\otimes(E_*E/p^r)\otimes L
 \arrow[dr,"\theta_u\otimes1\otimes1"] & \\
 (E_*E/p^r)\otimes L
 \arrow[rr,bend right=22,"1\otimes\Theta_u"']
 \arrow[r,shift left=1.1ex,"\Delta\otimes1"]
 \arrow[r,shift right=1.1ex,"1\otimes\psi_L"'] &
 (E_*E/p^r)\otimes(E_*E/p^r)\otimes L
 \arrow[r,"1\otimes\theta_u\otimes1"'] &
 (E_*E/p^r)\otimes L.
 \end{tikzcd}
\]
Coassociativity identifies each pair of parallel arrows after precomposition with $\psi_L$. Equation~\eqref{eq:Theta-centrality} then allows us to apply $\theta_u$ to either of the two cooperation factors after $\Delta$. Explicitly,
\[
 \begin{aligned}
 \psi_L\Theta_u
 &=\psi_L(\theta_u\otimes1)\psi_L
 &&\text{by \eqref{eq:Theta}}\\
 &=(\theta_u\otimes1\otimes1)(1\otimes\psi_L)\psi_L
 &&\text{by linearity of $\psi_L$}\\
 &=(\theta_u\otimes1\otimes1)(\Delta\otimes1)\psi_L
 &&\text{by coassociativity}\\
 &=(1\otimes\theta_u\otimes1)(\Delta\otimes1)\psi_L
 &&\text{by \eqref{eq:Theta-centrality}}\\
 &=(1\otimes\theta_u\otimes1)(1\otimes\psi_L)\psi_L
 &&\text{by coassociativity}\\
 &=(1\otimes\Theta_u)\psi_L
 &&\text{by \eqref{eq:Theta}.}
 \end{aligned}
\]
Thus $\Theta_u$ is a comodule map.

For a comodule map $g:L\to L'$, applying $\theta_u\otimes1$ to $\psi_{L'}g=(1\otimes g)\psi_L$ gives $\Theta_u g=g\Theta_u$. This proves naturality. The composition law follows from $[u]_F\circ[v]_F=[uv]_F$, and the inverse strict isomorphism $[u]_F^{-1}$ supplies the inverse operation. The unit identities also show that $\Theta_u$ acts as the identity on every internal shift of $E_*/p^r$.
\end{proof}

\begin{proposition}[Ghosts from formal multiplication]\label{prop:principal-unit-ext}
For every graded $E_*E/p^r$-comodule $L$ and every integer $u\equiv1\pmod{p^r}$, the operation $\Theta_u:L\to L$ induces the identity on $\Ext^{s,t}_{E_*E}(E_*,L)$ for all $s\geq0$ and $t\in\Z$. Consequently, $1-\Theta_u:L\to L$ is a ghost in $\D(E_*E)$.
\end{proposition}
\begin{proof}
Ravenel's coefficient change-of-rings comparison \cite[Proposition~A1.2.16(b)]{Ravenel} gives natural isomorphisms
\[
 \Ext^{s,t}_{E_*E}(E_*,L) \cong\Ext^{s,t}_{E_*E/p^r}(E_*/p^r,L).
\]
By Lemma~\ref{lem:principal-unit-automorphism}, $\Theta_u$ acts termwise on complexes and descends to a natural automorphism of the identity functor on $\D(E_*E/p^r)$, acting as the identity on every internal shift of $E_*/p^r$. Thus for any map $g:\Sigma^t(E_*/p^r)\to L[s]$, naturality gives the commutative square
\[
 \begin{tikzcd}[column sep=huge,row sep=large]
 \Sigma^t(E_*/p^r) \arrow[r,"g"]
 \arrow[d,"{\Theta_u=\id}"'] &
 L[s] \arrow[d,"{\Theta_u}"] \\
 \Sigma^t(E_*/p^r) \arrow[r,"g"'] & L[s].
 \end{tikzcd}
\]
Every Ext class in the comparison above is represented by such a map $g$. The square gives $\Theta_u\circ g=g$, so $\Theta_u$ induces the identity on these groups. Negative Ext groups from the unit to $L$ vanish since both lie in the heart, so $1-\Theta_u:L\to L$ is a ghost.
\end{proof}

This proves Proposition~\ref{prop:intro-criterion}. We now specialize to $E=BP$ and $r=1$, and construct a finite comodule on which $1-\Theta_{1+p}$ is nonzero.

\subsection{The finite coefficient comodule}
\label{subsec:coefficient-ghost}

Consider the free $BP_*/p$-module of rank two
\begin{equation}\label{eq:M}
 M=(BP_*/p)\{e_0,e_1\},\qquad |e_0|=0,\quad |e_1|=2p-2,
\end{equation}
with coaction
\begin{equation}\label{eq:M-coaction}
 \psi_M(e_0)=1\otimes e_0,\qquad \psi_M(e_1)=1\otimes e_1+t_1\otimes e_0.
\end{equation}
Since $\Delta(t_1)=t_1\otimes1+1\otimes t_1$ and $\epsilon(t_1)=0$, these formulas define a comodule structure.

To compute $\Theta_{1+p}$, recall that in characteristic $p$, the Araki generators satisfy $[p]_F(x)=v_1x^p+O(x^{p+1})$ \cite[equation~(A2.2.4)]{Ravenel}. Hence
\[
 \begin{aligned}
 [1+p]_F(x)&=F(x,[p]_F(x))=x+v_1x^p+O(x^{p+1}),\\
 [1+p]_F^{-1}(x)&=x-v_1x^p+O(x^{p+1}).
 \end{aligned}
\]
By the inverse-coordinate convention \cite[Lemma~A2.1.26]{Ravenel},
\begin{equation}\label{eq:theta-mod-p}
 \theta_{1+p}(t_1)=-v_1,\qquad \Theta_{1+p}(e_0)=e_0,\quad\Theta_{1+p}(e_1)=e_1-v_1e_0.
\end{equation}
Define
\begin{equation}\label{eq:coefficient-map}
 f=1-\Theta_{1+p}:M\longrightarrow M, \qquad f(e_0)=0,\quad f(e_1)=v_1e_0.
\end{equation}

\begin{mainthm}\label{thm:main}
The comodule $M$ belongs to $\D_{\mathrm{fin}}(BP_*BP)$ and can be built from four shifts of $BP_*$ by cofiber sequences. The map $f:M\to M$ is a nonzero ghost of exact additive order $p$, and
\[
 f^2=0,\qquad f^{\otimes^{\mathbb L}_{BP_*}r}\ne0 \quad(r\geq1).
\]
Each tensor power also has additive order $p$ and square zero under composition.
\end{mainthm}
Here $\otimes_{BP_*}^{\mathbb L}$ denotes the derived tensor product over $BP_*$. The map $f^{\otimes^{\mathbb L}_{BP_*}r}$ acts on $M^{\otimes^{\mathbb L}_{BP_*}r}$ by tensoring $r$ copies of $f$.

\begin{proof}
The inclusion and quotient give an exact sequence
\begin{equation}\label{eq:M-extension}
 0\longrightarrow BP_*/p\xrightarrow{1\mapsto e_0}M \xrightarrow{e_1\mapsto1}\Sigma^{2p-2}(BP_*/p)\longrightarrow0.
\end{equation}
Since $BP_*/p$ is the cofiber of multiplication by $p$ on $BP_*$, this constructs $M$ from four shifts of $BP_*$. Proposition~\ref{prop:principal-unit-ext} shows that $f$ is a ghost:
\begin{equation}\label{eq:ghost-calculation}
 \Ext^{s,t}_{BP_*BP}(BP_*,M) \xrightarrow{\ f_*=0\ }\Ext^{s,t}_{BP_*BP}(BP_*,M).
\end{equation}
Since $f(e_1)=v_1e_0\ne0$, it is nonzero as a comodule map and hence in the derived category. The relations $pf=0$ and $f^2=0$ hold already on $M$, so $f$ has exact additive order $p$.

For the tensor powers, cohomology in degree zero gives
\[
 H^0\bigl(M^{\otimes^{\mathbb L}_{BP_*}r}\bigr) \cong M^{\otimes_{BP_*}r}.
\]
This tensor product is free over $BP_*/p$ on the $2^r$ tensors in $e_0,e_1$. The induced map sends
\[
 e_1^{\otimes r}\longmapsto v_1^r e_0^{\otimes r}\ne0,
\]
since $BP_*/p=\F_p[v_1,v_2,\ldots]$. Thus the derived map is nonzero. The relations $p\cdot f=0$ and $f^2=0$ give the order and composition assertions for every tensor power.
\end{proof}

\begin{remark}\label{rem:classical-realization}
As a comodule, $M$ is also realized by the classical finite spectrum $S/(p,\alpha_1)$, where $S$ is the $p$-local sphere, $\alpha_1\in\pi_{2p-3}S$ is the first alpha-family element, and $S/(p,\alpha_1)=S/p\wedge\cofib(\alpha_1)$ \cite[Theorems~2.3.4 and~5.3.7]{Ravenel}. With compatible cell orientations, $BP_*(\cofib(\alpha_1))$ has coaction \eqref{eq:M-coaction} over $BP_*$. This $BP_*$-module is free, so smashing with $S/p$ reduces its homology modulo $p$. At $p=2$, the class $\alpha_1$ is $\eta$.
\end{remark}

\subsection{Endomorphisms}\label{sec:endomorphisms}

The map $f$ accounts for every degree-zero ghost endomorphism of $M$.

\begin{proposition}\label{prop:end-ring}
The endomorphism ring of $M$ in cohomological and internal degree zero is
\[
 \End^{0,0}_{\D(BP_*BP)}(M) \cong\F_p[\varepsilon]/(\varepsilon^2), \qquad f\longleftrightarrow\varepsilon.
\]
Its ideal of ghosts is $(\varepsilon)$.
\end{proposition}
\begin{proof}
Since $M$ lies in the heart, its degree-zero derived endomorphisms are its ordinary comodule endomorphisms. A degree-zero $BP_*/p$-linear endomorphism has the form
\[
 g(e_0)=ae_0,\qquad g(e_1)=be_1+cv_1e_0, \qquad a,b,c\in\F_p.
\]
Comparison of coactions gives $(b-a)\cdot t_1=0$, hence $a=b$. Conversely, every map with $a=b$ respects the coaction, since $v_1$ is invariant modulo $p$. Thus every endomorphism is uniquely $a\id_M+cf$, and $f^2=0$ gives the ring formula. A ghost must kill the map $BP_*\to M$ given by $1\mapsto e_0$, so $a=0$. Conversely, every multiple of $f$ is a ghost by \eqref{eq:ghost-calculation}.
\end{proof}

The ring is local, so $M$ is indecomposable. For $c\in\F_p^\times$, the maps $1+cf$ are nonidentity self-equivalences of order $p$ and induce the identity on all detecting Ext groups.

\begin{remark}
The tensor nonnilpotence is compatible with the weak algebraic nilpotence theorems of Barthel--Heard~\cite[Section~5]{BH}, whose hypotheses use chromatic detecting objects. After derived extension of scalars to $v_1^{-1}BP_*/p$, degree-zero cohomology is the ordinary extension of scalars of $M$. The induced map still sends $e_1$ to $v_1e_0\ne0$. Thus $f$ remains nonzero after this chromatic base change, although it acts trivially on all detecting Ext groups.
\end{remark}

\subsection{Projective spaces and composites of ghosts}
\label{subsec:projective-spaces}

For any $r,n\geq1$, finite projective spaces will provide a ghost whose first $n$ powers all have exact additive order $p^r$. Let $E$ be as in Section~\ref{subsec:strict-operation}, and assume that $E$ is not rational, equivalently $E_*/p\ne0$. Write $D$ for Spanier--Whitehead duality in the $p$-local stable homotopy category.

\begin{definition}\label{def:projective-comodule}
For $m,r\geq1$, define the $E_*E/p^r$-comodule
\[
 M(m,r) =E_*\bigl(D(\Sigma^\infty\mathbb{C}P^m)\wedge S/p^r\bigr).
\]
Here $\mathbb{C}P^m$ is based at $\mathbb{C}P^0$, and its suspension spectrum is reduced. Write $x$ for the orientation class, viewed in internal degree $-2$. Finite duality and the complex orientation identify
\begin{equation}\label{eq:projective-comodule}
 M(m,r) \cong\widetilde E^{-*}(\mathbb{C}P^m)/p^r =(E_*/p^r)\{x,x^2,\ldots,x^m\}, \qquad |x^j|=-2j.
\end{equation}
Formal power series are evaluated at $x$ using $x^{m+1}=0$.
\end{definition}

\begin{lemma}\label{lem:projective-operation}
For every integer $u\equiv1\pmod{p^r}$, the operation $\Theta_u$ on $M(m,r)$ satisfies
\begin{equation}\label{eq:projective-operation}
 \Theta_u(x) =[u]_F^{-1}(x).
\end{equation}
Moreover, $M(m,r)$ is built from $2m$ shifts of $E_*$ and satisfies $p^r\cdot\id_{M(m,r)}=0$.
\end{lemma}
\begin{proof}
Under finite duality, the coaction is the composite
\[
 \psi:E^{-*}(\mathbb{C}P^m) \xrightarrow{(\eta_L)_*}(E\wedge E)^{-*}(\mathbb{C}P^m) \cong E_*E\otimes_{E_*}E^{-*}(\mathbb{C}P^m).
\]
The first map sends $x$ to $\eta_L(x)$. The displayed isomorphism is the K\"unneth isomorphism using the right $E$-factor, so $\eta_R(x)$ corresponds to $1\otimes x$. These two orientations have formal group laws $\eta_L^*F$ and $\eta_R^*F$, respectively. The universal strict isomorphism $h:\eta_L^*F\to\eta_R^*F$ satisfies $h(\eta_L(x))=\eta_R(x)$ \cite[Theorems~4.1.11 and~4.1.19]{Ravenel}. Thus, passing to reduced cohomology and reducing modulo $p^r$ gives
\[
 \psi(x) =h^{-1}(1\otimes x).
\]

By Definition~\ref{def:formal-multiplication-operation}, $\theta_u(h)=[u]_F$. Specialization commutes with composition, so $\theta_u(h^{-1})=[u]_F^{-1}$. Applying $\theta_u\otimes1$ to the coaction therefore proves \eqref{eq:projective-operation}.

For finiteness, use the even-cell filtration of $M(m,r)$. Its comodule quotients are $\Sigma^{-2j}(E_*/p^r)$ for $1\leq j\leq m$. Each is the cofiber of $p^r$ on a shifted unit, so $M(m,r)$ is built from $2m$ shifts of $E_*$. Equation~\eqref{eq:projective-comodule} also gives $p^r\cdot\id_{M(m,r)}=0$.
\end{proof}

\begin{example}\label{ex:projective-rank-two}
For $m=p$ and $r=1$, the construction recovers the rank-two comodule $M$ of Section~\ref{subsec:coefficient-ghost} as a direct summand after extension of scalars:
\[
 \Sigma^{2p}M(p,1) \cong (E_*\otimes_{BP_*}M) \oplus\bigoplus_{j=2}^{p-1}\Sigma^{2p-2j}(E_*/p).
\]
Indeed, the initial terms $x+t_1x^p$ of the inverse strict isomorphism give $e_0=x^p$ and $e_1=x$ on the first summand; all remaining generators are primitive. The map $1-\Theta_{1+p}$ sends $e_1$ to $v_1e_0$ and is zero on the other generators. It therefore restricts to the coefficient extension of $f$. At $p=2$, the displayed direct sum has no additional summands.
\end{example}

The following theorem proves Theorem~\ref{thm:algebraic-intro}\textup{(1)} by taking $E=BP$.

\begin{theorem}\label{thm:projective-composites}
Let $E$ be an even, $p$-local Landweber exact theory with a $p$-typical orientation, and assume $E_*/p\ne0$. For every $r,n\geq1$, there is an $m\geq1$ such that the ghost
\[
 f=1-\Theta_{1+p^r}: M(m,r)\longrightarrow M(m,r)
\]
satisfies
\[
 p^r\cdot\id_{M(m,r)}=0,\qquad p^{r-1}\cdot f^j\ne0\quad(1\leq j\leq n).
\]
Thus each of $f,f^2,\ldots,f^n$ has exact additive order $p^r$. In particular, at every prime the ghost ideal in $\D_{\mathrm{fin}}(BP_*BP)$ is not nilpotent, even when restricted to heart objects annihilated by $p$.
\end{theorem}
\begin{proof}
Proposition~\ref{prop:principal-unit-ext} shows that $f$ is a ghost for every $m$. By increasing $n$, we may assume that it is a power of $p$. Since $p$ is regular on $E_*$, multiplication by $p^{r-1}$ induces an isomorphism $E_*/p\cong p^{r-1}(E_*/p^r)$. The free basis in \eqref{eq:projective-comodule} gives the corresponding isomorphism $M(m,r)/p\cong p^{r-1}M(m,r)$, commuting with $f$. It therefore suffices to make $f^n$ nonzero modulo $p$.

Write
\[
 1-(1+p^r)^{-n}=cp^k,\qquad c\in\Z_{(p)}^\times,
\]
and set $m=p^k$. We calculate modulo $p$, using $x^{p^k+1}=0$. The binomial theorem gives
\[
 f^n=(1-\Theta_{1+p^r})^n \equiv1-\Theta_{1+p^r}^{\,n}\pmod p.
\]
The composition law in Lemma~\ref{lem:principal-unit-automorphism} gives $\Theta_{1+p^r}^{\,n}=\Theta_{(1+p^r)^n}$. By Lemma~\ref{lem:projective-operation}, this operation sends $x$ to $[(1+p^r)^{-n}]_F(x)$. Hence
\begin{equation}\label{eq:projective-power}
 f^n(x)\equiv x-[(1+p^r)^{-n}]_F(x) =x-[1-cp^k]_F(x)\pmod p.
\end{equation}
To compute $[1-cp^k]_F(x)$, we first iterate the $p$-series. The expansion $[p]_F(x)=v_1x^p+O(x^{p+1})$ gives $[p^{k-1}]_F(x)\in(x^{p^{k-1}})$ by induction. After substitution, the remainder lies in $(x^{(p+1)p^{k-1}})=0$, since $(p+1)p^{k-1}\ge p^k+1$. Iterating the leading term therefore gives
\begin{equation}\label{eq:projective-p-series}
 \begin{aligned}
 [p^k]_F(x)
 &=v_1\bigl([p^{k-1}]_F(x)\bigr)^p\\
 &=v_1^{1+p+\cdots+p^{k-1}}x^{p^k}\\
 &=v_1^{(p^k-1)/(p-1)}x^{p^k}.
 \end{aligned}
\end{equation}
Since $[p^k]_F(x)\in(x^{p^k})$ and $x^{p^k+1}=0$, we have $x[p^k]_F(x)=\bigl([p^k]_F(x)\bigr)^2=0$. Thus $F(X,Y)\equiv X+Y\pmod{XY}$ and $[-c]_F(Y)=-cY+O(Y^2)$ give
\begin{equation}\label{eq:projective-scalar-expansion}
 [1-cp^k]_F(x) =F\bigl(x,[-c]_F([p^k]_F(x))\bigr) =x-c[p^k]_F(x).
\end{equation}
Combining \eqref{eq:projective-power}, \eqref{eq:projective-p-series}, and \eqref{eq:projective-scalar-expansion} gives
\[
 f^n(x)\equiv c[p^k]_F(x) =cv_1^{(p^k-1)/(p-1)}x^{p^k}\pmod p.
\]
Landweber exactness makes $v_1$ regular on the nonzero ring $E_*/p$. Since $c$ is a unit, $f^n$ is nonzero modulo $p$, and hence $p^{r-1}f^n\ne0$.

It follows that $p^{r-1}f^j\ne0$ for $1\leq j\leq n$. Since $p^r\id_{M(p^k,r)}=0$, each power has exact additive order $p^r$, also in the derived category because the source and target lie in the heart.
\end{proof}

\begin{example}\label{ex:order-four-composite}
At $p=2$, take $E=BP$ and $f=1-\Theta_5:M(8,2)\to M(8,2)$. Since $1-5^{-2}=(3/25)2^3$, the calculation in the proof gives
\[
 f^2(x)\equiv v_1^7x^8\not\equiv0\pmod2, \qquad 2\cdot f^2\ne0.
\]
Thus both $f$ and $f^2$ have exact additive order four. The source can be built from sixteen shifts of the unit, and its motivic realization has a model with thirty-two cells over the uncompleted sphere by Proposition~\ref{prop:projective-realization}.
\end{example}

\begin{remark}\label{rem:classical-principal-units}
The automorphisms $\Theta_u$ of $M(m,r)$ admit classical realizations. Let $u$ be a positive integer prime to $p$, and consider
\[
 q_u:\mathbb{C}P^m\longrightarrow\mathbb{C}P^m,\qquad [z_0:\cdots:z_m]\longmapsto[z_0^u:\cdots:z_m^u].
\]
The restriction of $q_u$ to $\mathbb{C}P^1$ has degree $u$. Since $H^*(\mathbb{C}P^m;\Z_{(p)})$ is generated in degree two, $q_u^*$ acts by $u^j$ on $H^{2j}$. Thus $q_u$ is a $p$-local homotopy equivalence. For $u>1$, its action on $H^2$ shows that it is not homotopic to the identity, even $p$-locally.

Now assume $u\equiv1\pmod{p^r}$, so $q_u$ induces the identity on $H^*(\mathbb{C}P^m;\Z/p^r)$. On $E$-cohomology, $q_u^*(x)=[u]_F(x)$, so
\[
 g_u=D(\Sigma^\infty q_u^{-1})\wedge\id_{S/p^r}: D(\Sigma^\infty\mathbb{C}P^m)\wedge S/p^r \longrightarrow D(\Sigma^\infty\mathbb{C}P^m)\wedge S/p^r
\]
realizes $\Theta_u$, where the inverse is taken $p$-locally. For $E=BP$, the difference $1-g_u$ induces zero on the Adams--Novikov $E_2$-term. This does not imply that $1-g_u$ is a classical ghost.
\end{remark}

\begin{remark}
The ghost lemma~\cite[Theorem~3.5]{Christensen} bounds the length of a nonzero composite of ghosts with any fixed finite source. Theorem~\ref{thm:projective-composites} shows that no bound works uniformly over all finite sources.
\end{remark}

\subsection{Local algebraic ghosts}\label{subsec:local-algebraic}

We now adapt the construction to height $h$, using the first nonzero term of the $p$-series modulo $I_h$. Fix a height $h>0$ and an integer $r\geq1$, and put
\[
 I_h=(p,v_1,\ldots,v_{h-1})\subset E(h)_*,\qquad E(h)_*/I_h=\F_p[v_h^{\pm1}].
\]
Change of rings along $BP_*\to E(h)_*$ gives an equivalence \cite[Theorem~6.4(2) and Definition~6.13]{BH}
\begin{equation}\label{eq:local-presentation}
 L_{K(h)}\operatorname{Stable}(BP_*BP) \simeq L_{K(h)}\operatorname{Stable}(E(h)_*E(h)).
\end{equation}
Here $L_{K(h)}$ denotes algebraic localization. In the $E(h)$ presentation, it is derived $I_h$-adic completion \cite[Corollary~5.26]{BHV}, giving
\[
 L_{K(h)}E(h)_*\simeq \operatorname*{holim}_{s\geq1}E(h)_*/I_h^s.
\]
The inverse equivalence sends $E(h)_*/I_h^s$ to $v_h^{-1}(BP_*/I_h^s)$ \cite[Propositions~5.22 and~7.16]{BHV}, since $p,v_1,\ldots,v_{h-1}$ act nilpotently on $BP_*/I_h^s$. Thus
\[
 L_{K(h)}BP_*\simeq \operatorname*{holim}_{s\geq1}v_h^{-1}(BP_*/I_h^s).
\]
Both limits are taken in the corresponding stable comodule categories. Thick and localizing subcategories include all even internal shifts.

For $m\geq1$, consider the comodule
\[
 N_m=\widetilde E(h)^{-*}(\mathbb{C}P^m)/I_h^r =(E(h)_*/I_h^r)\{x,x^2,\ldots,x^m\}, \qquad |x^j|=-2j.
\]
The following theorem proves Theorem~\ref{thm:algebraic-intro}\textup{(2)}.

\begin{theorem}\label{thm:local-composites}
For every prime $p$ and integers $h,r,n\geq1$, there is a compact object $M\in L_{K(h)}\operatorname{Stable}(BP_*BP)$ with $p^r\cdot\id_M=0$ and a ghost $f:M\to M$ such that $f^j$ has exact additive order $p^r$ for every $1\leq j\leq n$. Consequently, the local algebraic generating hypothesis fails at every positive height, and the ghost ideal on compact local objects is not nilpotent.
\end{theorem}
\begin{proof}
Work in $\operatorname{Stable}(E(h)_*E(h))$. The regular sequence $p,v_1,\ldots,v_{h-1}$, whose terms are invariant modulo their predecessors, makes $E(h)_*/I_h$ compact. Since $N_m$ is finitely presented over $E(h)_*$ and annihilated by $I_h^r$, the Landweber filtration theorem \cite[Theorem~5.7]{HS} gives a finite filtration by shifts of $E(h)_*/I_h$. Hence $N_m\in\Thick(E(h)_*/I_h)$, the compact subcategory of the local category \cite[Theorem~3.5.3]{HPS}.

Localization adjunction and the bounded comparison \cite[Proposition~3.6 and Lemma~3.16]{BH} identify maps from shifts of the local unit to $N_m$ with $\Ext^{*,*}_{E(h)_*E(h)}(E(h)_*,N_m)$. For $u=1+p^r$, Proposition~\ref{prop:principal-unit-ext} therefore shows that $f=1-\Theta_u:N_m\to N_m$ is a local ghost.

By increasing $n$, we may assume that it is a power of $p$. Write
\[
 1-u^{-n}=1-(1+p^r)^{-n}=cp^k,\qquad c\in\Z_{(p)}^\times,
\]
and take $m=p^{hk}$. Modulo $I_h$, the Araki formula \cite[equation~(A2.2.4)]{Ravenel} gives $[p]_F(x)=v_hx^{p^h}+O(x^{p^h+1})$. The height-$h$ analogue of the calculation in Theorem~\ref{thm:projective-composites} gives
\[
 f^n(x)\equiv cv_h^{(p^{hk}-1)/(p^h-1)}x^{p^{hk}}\pmod{I_h}.
\]
Since $I_h/I_h^r$ is nilpotent, the $x^{p^{hk}}$-coefficient is a unit in $E(h)_*/I_h^r$. Also, $p^{r-1}\ne0$ in this ring, as setting $v_1,\ldots,v_{h-1}$ to zero gives $(\Z/p^r)[v_h^{\pm1}]$. Hence $p^{r-1}f^n\ne0$. It follows that $p^{r-1}f^j\ne0$ for $1\leq j\leq n$. Since $p^r\id_{N_m}=0$, each such power has exact additive order $p^r$. Full faithfulness of the bounded comparison preserves these orders, and transport along \eqref{eq:local-presentation} completes the proof.
\end{proof}

\begin{remark}
Every compact local object lies in $\Thick(E(h)_*/I_h)$. Localizing the finite construction of this quotient from $E(h)_*$ places it in the thick subcategory generated by shifts of the local unit. The ghost lemma \cite[Theorem~3.5]{Christensen} therefore bounds composites out of any fixed compact local object, although the unit need not be compact. The theorem shows that these bounds are not uniform.
\end{remark}

\section{Intrinsic motivic realizations}\label{sec:motivic}

We prove Theorem~\ref{thm:failure}\textup{(1)} by realizing algebraic ghosts as $C\tau$-linear motivic ghosts. Theorem~\ref{thm:intrinsic} gives the eight-cell example, and Proposition~\ref{prop:projective-realization} gives the general projective-space families.

\subsection{The complex-motivic realization}\label{subsec:complex}

Let $\alpha_1\in\pi_{2p-3,p-1}C\tau$ be the class represented by the cobar cocycle $[t_1]$. At $p=2$, this is $\eta$. Form the $C\tau$-module cofiber sequence
\begin{equation}\label{eq:motivic-triangle-intro}
 \Sigma^{2p-3,p-1}(C\tau/p)\xrightarrow{\alpha_1}C\tau/p \xrightarrow{i}X\xrightarrow{q}\Sigma^{2p-2,p-1}(C\tau/p), \qquad X=C\tau/(p,\alpha_1).
\end{equation}
The invariant $v_1$ modulo $p$ defines a $C\tau$-linear map $v_1:\Sigma^{2p-2,p-1}(C\tau/p)\to C\tau/p$.

\begin{mainthm}\label{thm:intrinsic}
Under the special-fiber comparison, the map $f$ of Theorem~\ref{thm:main} corresponds to the $C\tau$-linear composite
\begin{equation}\label{eq:motivic-map-intro}
 \alpha:X\xrightarrow{q}\Sigma^{2p-2,p-1}(C\tau/p) \xrightarrow{v_1}C\tau/p\xrightarrow{i}X.
\end{equation}
Its underlying motivic map is a nonzero ghost of exact additive order $p$ and square zero. The underlying spectrum has a model with eight cells over the uncompleted sphere
\begin{equation}\label{eq:sphere-model-intro}
 X\simeq\one_{\C}/(p,\widetilde\alpha_1,\tau).
\end{equation}
\end{mainthm}

Here $\widetilde\alpha_1\in\pi_{2p-3,p-1}\one_{\C}$ lifts $\alpha_1$. Take $\eta$ at $p=2$. At odd primes, take the motivic $J$-image of a Bott generator~\cite{HKO-J}, rescaled by an integer prime to $p$ so that its completed Adams--Novikov class is $[t_1]$ \cite[Proposition~5.9 and Section~5.3]{Stahn}.

The four $C\tau$-cells of $X$ have bidegrees
\begin{equation}\label{eq:R-cells}
 (0,0),\quad(1,0),\quad(2p-2,p-1),\quad(2p-1,p-1).
\end{equation}
Figure~\ref{fig:R-cells} describes $\alpha=iv_1q$ in these cells. Its eight-cell model over the sphere has the four bidegrees in \eqref{eq:R-cells} and their translates by $(1,-1)$. At the prime two, the model is $\one_{\C}/(2,\eta,\tau)$.

\begin{proof}

Work first over $\C$ in the $p$-complete cellular category. The symmetric monoidal equivalence between $C\tau$-modules and $\operatorname{Stable}(BP_*BP)$ recalled in the introduction \cite[Corollary~1.2 and Remark~4.15]{GWX} sends, with our shift convention,
\begin{equation}\label{eq:comparison-shifts}
 (\Sigma^{2b}BP_*)[a-2b]\longleftrightarrow\Sigma^{a,b}C\tau.
\end{equation}
If a $C\tau$-module $T$ corresponds to a degreewise $p$-completed comodule $L$, then
\begin{equation}\label{eq:comparison-homotopy}
 \pi_{a,b}T \cong\Ext^{2b-a,2b}_{BP_*BP}(BP_*,L).
\end{equation}
Since $pM=0$, degreewise $p$-completion leaves $M$ and $f$ unchanged.

With cobar differential $\partial(z)=1\otimes z-\psi(z)$, the connecting class of \eqref{eq:M-extension} is $-[t_1]$, so the attaching class is $\alpha_1$. Hence
\begin{equation}\label{eq:realization-M}
 M\longleftrightarrow X=C\tau/(p,\alpha_1).
\end{equation}
Formula \eqref{eq:coefficient-map} then identifies $f$ with $\alpha=iv_1q$.

Free--forgetful adjunction gives
\begin{equation}\label{eq:free-forgetful}
 [\Sigma^{a,b}C\tau,X]_{C\tau} \cong[\Sigma^{a,b}\widehat\one_{\C},X] \cong[\Sigma^{a,b}\one_{\C},X].
\end{equation}
Consequently \eqref{eq:ghost-calculation} proves that the underlying map $\alpha$ is a motivic ghost.

To prove nonvanishing, recall that the heart equivalence is induced by completed $\mathrm{MGL}_{\C}$-homology of the underlying spectrum, followed by change of rings from $MU_*MU$ to $BP_*BP$ \cite[Section~1.2 and Proposition~4.11]{GWX}. In this description, $\alpha$ sends
\[
 e_1\longmapsto v_1e_0\ne0.
\]
Thus the underlying map is nonzero. Since $pM=0$ and $f^2=0$, the comparison gives $p\,\id_X=0$ and $\alpha^2=0$, so $\alpha$ has exact additive order $p$.

Finally, in the uncompleted category, form the cofiber of the action of $\widetilde\alpha_1$ on $\one_{\C}/p$:
\[
 T=\cofib\bigl(\widetilde\alpha_1: \Sigma^{2p-3,p-1}(\one_{\C}/p) \longrightarrow\one_{\C}/p\bigr).
\]
The spectrum $T$ has four cells over the sphere and is $p$-complete by Lemma~\ref{lem:descent}. Hence $\widehat\one_{\C}\wedge T\simeq T$, and the completed scalar $\tau$ defines a map in the uncompleted category $\tau_T:\Sigma^{0,-1}T\to T$. Its cofiber satisfies
\[
 \cofib(\tau_T)\simeq C\tau\wedge_{\widehat\one_{\C}}T \simeq C\tau/(p,\alpha_1)=X.
\]
This gives the eight-cell model in \eqref{eq:sphere-model-intro} and proves the case $r=n=1$ of Theorem~\ref{thm:failure}\textup{(1)}.
\end{proof}

\begin{corollary}\label{cor:relative-tensor-nonnilpotence}
For every $r\geq1$, the relative smash power
\[
 \alpha^{\wedge_{C\tau}r}: X^{\wedge_{C\tau}r}\longrightarrow X^{\wedge_{C\tau}r}
\]
is nonzero in the category of $C\tau$-modules. Each such map has additive order $p$ and square zero under composition.
\end{corollary}
\begin{proof}
The special-fiber equivalence is symmetric monoidal \cite[Remark~4.15]{GWX}. The nonzero degree-zero cohomology map in the proof of Theorem~\ref{thm:main} is unchanged by degreewise $p$-completion. The equivalence therefore carries the nonzero algebraic tensor power to the displayed relative smash power. The additive-order and square-zero relations are also preserved.
\end{proof}

\begin{remark}\label{rem:intrinsic-self-equivalences}
The special-fiber equivalence identifies the degree-zero $C\tau$-linear endomorphism ring of $X$ with the dual-number ring of Proposition~\ref{prop:end-ring}. For $c\in\F_p^\times$, the maps $1+c\alpha$ are nonidentity self-equivalences of order $p$ which induce the identity on all bigraded homotopy groups. These assertions follow from $\alpha^2=0$, the exact additive order of $\alpha$, and the fact that $\alpha$ is a ghost.
\end{remark}

\begin{remark}\label{rem:betti}
The complex Betti realization $\operatorname{Re}(X)$ is contractible. Indeed, $p$-completed Betti realization inverts $\tau$, whereas $p\cdot\id_{X}=0$ implies that the ordinary Betti realization of $X$ is already $p$-complete. Thus the motivic counterexample gives no counterexample to Freyd's classical generating hypothesis.
\end{remark}

\subsection{Motivic realizations of the projective-space families}
\label{subsec:motivic-projective}

We now realize the projective-space families of Section~\ref{subsec:projective-spaces}, with $E=BP$.

\begin{proposition}\label{prop:projective-realization}
For $m,r\geq1$, let $X$ be the $C\tau$-module corresponding to $M(m,r)$. Then $X$ is annihilated by $p^r$ and has a model with $4m$ cells over the uncompleted sphere $\one_{\C}$. Every ghost endomorphism of this coefficient comodule realizes to a $C\tau$-linear motivic ghost. The additive order of every composition power is preserved after forgetting the module structure.
\end{proposition}
\begin{proof}
Under the comparison, the filtration in Section~\ref{subsec:projective-spaces} has quotients $\Sigma^{-2j,-j}(C\tau/p^r)$ for $1\leq j\leq m$. Each has a four-cell model over the uncompleted sphere by Section~\ref{sec:complex}; refining the filtration by these cell structures gives a model with $4m$ cells.

Degreewise completion leaves $M(m,r)$ and its endomorphisms unchanged, so the identity $p^r\cdot\id_{M(m,r)}=0$ holds on its realization. The adjunction~\eqref{eq:free-forgetful} and Ext identification~\eqref{eq:comparison-homotopy} prove the ghost assertion. As explained in Section~\ref{subsec:complex}, the heart comparison is induced by completed $\mathrm{MGL}_{\C}$-homology of the underlying spectrum. It therefore detects every nonzero integer multiple of every composition power, proving that additive orders are preserved after forgetting the module structure.
\end{proof}

For $r,n\geq1$, choose $m$ as in Theorem~\ref{thm:projective-composites}. Its ghost on $M(m,r)$ realizes to a $C\tau$-linear ghost $\alpha:X\to X$, and Proposition~\ref{prop:projective-realization} gives
\[
 p^r\cdot \alpha^j=0,\qquad p^{r-1}\cdot \alpha^j\ne0 \qquad(1\leq j\leq n).
\]
Together with the eight-cell example in Theorem~\ref{thm:intrinsic}, this completes the proof of Theorem~\ref{thm:failure}\textup{(1)}. All these spectra become contractible after inverting $\tau$.

\begin{remark}\label{rem:cobordism-contrast}
Ordinary algebraic cobordism homology $\mathrm{MGL}^{\C}_{**}$ distinguishes the Bockstein construction of Section~\ref{sec:complex} from the intrinsic ghost of Theorem~\ref{thm:intrinsic}:
\[
 \mathrm{MGL}^{\C}_{**}(\beta_{p^r})=0,\qquad \mathrm{MGL}^{\C}_{**}(\alpha)\ne0.
\]
For $Y=\one_{\C}/(p^r,\tau)$, with $r\geq1$, we have $\mathrm{MGL}^{\C}_{**}(Y)\cong MU_*/p^r$ in Chow degree zero \cite[Section~1.2, p.~325]{GWX}. The target $\Sigma^{1,-1}Y$ has homology in Chow degree three, proving the first equality. The second was proved in Section~\ref{subsec:complex}. The first equality concerns the induced homology map; it does not assert that $\mathrm{MGL}_{\C}\wedge\beta_{p^r}$ is null-homotopic.
\end{remark}

\section{Homotopy modules and nonfullness}
\label{sec:fullness-classification}

We prove Theorem~\ref{thm:fullness-classification} and Corollary~\ref{cor:homotopy-not-full} using the projective-space families and the following general observation about composites of two ghosts.

Let $\mathcal T$ be a symmetric monoidal triangulated category with unit $\one_{\mathcal T}$ and compatible bigraded shifts $\Sigma^{a,b}$, $(a,b)\in\Z^2$, where $\Sigma^{1,0}=\Sigma$ is the triangulated suspension and $\Sigma^{a,b}X=(\Sigma^{a,b}\one_{\mathcal T})\otimes X$. Assume $\mathcal T=\Thick\{\Sigma^{a,b}\one_{\mathcal T}:a,b\in\Z\}$. Define
\[
 \pi_{a,b}X=[\Sigma^{a,b}\one_{\mathcal T},X].
\]
The tensor product makes $\pi_{*,*}X$ a bigraded $\pi_{*,*}\one_{\mathcal T}$-module. Since the unit shifts generate $\mathcal T$, $\pi_{*,*}$ detects isomorphisms. A map is a ghost if it induces zero on $\pi_{*,*}$.

\begin{lemma}\label{lem:double-ghost-classification}
Let $\alpha:X\to Y$ and $\beta:Y\to Z$ be ghosts in $\mathcal T$ such that $\beta\alpha\ne0$. Set
\[
 A=\cofib(\beta\alpha:X\to Z),\qquad B=Z\oplus\Sigma X.
\]
Then $\pi_{*,*}A\cong\pi_{*,*}B$ as bigraded $\pi_{*,*}\one_{\mathcal T}$-modules. If $[X,Z]$ and $[X,\Sigma X]$ are finite, then $A\not\simeq B$.
\end{lemma}
\begin{proof}
We recall the splitting argument of Pirashvili--Redondo~\cite[proof of Theorem~1]{PirashviliRedondo}. Complete $\alpha$ and $\beta\alpha$ to triangles and use the triangle axioms to obtain a commutative diagram
\[
\begin{tikzcd}[column sep=large,row sep=large]
 X \arrow[r,"\alpha"] \arrow[d,equal]
 & Y \arrow[r,"j"] \arrow[d,"\beta"]
 & C \arrow[r,"q"] \arrow[d,"c"]
 & \Sigma X \arrow[d,equal] \\
 X \arrow[r,"\beta\alpha"]
 & Z \arrow[r,"i"]
 & A \arrow[r,"\delta"]
 & \Sigma X.
\end{tikzcd}
\]
Applying $\pi_{*,*}$ gives exact rows in the following diagram:
\begin{equation}\label{eq:double-ghost-extension}
\begin{tikzcd}[column sep=large,row sep=large]
 0 \arrow[r]
 & \pi_{*,*}Y \arrow[r,"j_*"] \arrow[d,"0"']
 & \pi_{*,*}C \arrow[r,"q_*"] \arrow[d,"c_*"]
 & \pi_{*,*}(\Sigma X) \arrow[r] \arrow[d,equal] \arrow[dl,dashed]
 & 0 \\
 0 \arrow[r]
 & \pi_{*,*}Z \arrow[r,"i_*"]
 & \pi_{*,*}A \arrow[r,"\delta_*"]
 & \pi_{*,*}(\Sigma X) \arrow[r]
 & 0.
\end{tikzcd}
\end{equation}
Since the left vertical map is zero, the dashed arrow exists and splits the lower row. Hence $\pi_{*,*}A\cong\pi_{*,*}B$.

Applying $[X,-]$ to the lower triangle gives
\[
 [X,X]\longrightarrow[X,Z]\longrightarrow[X,A] \longrightarrow[X,\Sigma X]\longrightarrow[X,\Sigma Z].
\]
The first map sends $\id_X$ to $\beta\alpha\ne0$. If $[X,Z]$ and $[X,\Sigma X]$ are finite, exactness therefore gives
\[
 \lvert[X,A]\rvert <\lvert[X,Z]\rvert\,\lvert[X,\Sigma X]\rvert =\lvert[X,B]\rvert.
\]
Consequently $A\not\simeq B$.
\end{proof}

We now use Lemma~\ref{lem:double-ghost-classification} to prove Theorem~\ref{thm:fullness-classification} and Corollary~\ref{cor:homotopy-not-full}.

\Needspace{4\baselineskip}
\begin{proof}[Proof of Theorem~\ref{thm:fullness-classification}]
Take $E=BP$, $r=1$, and $n=2$ in Theorem~\ref{thm:projective-composites}. For a suitable $m$, the finite comodule $M=M(m,1)$ supports a ghost
\[
 f=1-\Theta_{1+p}:M\longrightarrow M,\qquad f^2\ne0.
\]
By \eqref{eq:projective-comodule}, its underlying module is
\[
 M=(BP_*/p)\{x,x^2,\ldots,x^m\},\qquad |x^j|=-2j.
\]
Each fixed bidegree of the normalized cobar complex computing $\Ext^{*,*}_{BP_*BP}(BP_*,M)$ is finite-dimensional over $\F_p$: $M$ has finitely many generators and is bounded below internally, while the polynomial generators $v_i,t_i$ have positive degrees tending to infinity. Thus every $\Ext^{s,t}_{BP_*BP}(BP_*,M)$ is finite.

If $\pi_{*,*}X$ is finite in every bidegree, then $[K,X]$ is finite for every finite $K$, by induction on unit shifts, cofiber sequences, and retracts. Applied algebraically, this makes $[M,M]$ and $[M,M[1]]$ finite, so Lemma~\ref{lem:double-ghost-classification}, with both ghosts equal to $f$, proves \textup{(1)}.

For \textup{(2)} and \textup{(3)}, Proposition~\ref{prop:projective-realization} realizes $f$ as a motivic ghost $\alpha:X\to X$ with $\alpha^2\ne0$, where $X$ is finite over $\one_{\C}$ and $p\,\id_X=0$. By Lemma~\ref{lem:descent}, $X$ is $p$-complete, so completing its finite cellular model makes it finite over $\widehat\one_{\C}$ as well. The comparisons~\eqref{eq:comparison-homotopy} and~\eqref{eq:free-forgetful} give
\[
 [\Sigma^{a,b}\widehat\one_{\C},X] \cong [\Sigma^{a,b}\one_{\C},X] \cong \Ext^{2b-a,2b}_{BP_*BP}(BP_*,M).
\]
These groups are finite, and $\alpha$ induces zero on them. The preceding finiteness argument therefore makes $[X,X]$ and $[X,\Sigma^{1,0}X]$ finite in both motivic categories. Lemma~\ref{lem:double-ghost-classification}, with both ghosts equal to $\alpha$, now proves \textup{(2)} and \textup{(3)}.
\end{proof}

\begin{proof}[Proof of Corollary~\ref{cor:homotopy-not-full}]
In each setting, the homotopy functor detects isomorphisms. Fullness would therefore lift the module isomorphism in Theorem~\ref{thm:fullness-classification} to an equivalence of the corresponding objects, a contradiction.
\end{proof}

\section{Homological consequences}
\label{sec:homological-consequences}

We prove Corollary~\ref{cor:homological-consequences} by applying homological dimension bounds to the composites of ghosts constructed above. We first recall the dimensions used in the statement.

\begin{definition}[Homological dimensions]\label{def:homological-dimensions}
Work with graded left modules over a possibly bigraded ring $R$ and degree-preserving maps. The \emph{projective}, \emph{flat}, and \emph{injective dimensions} of $N$, denoted $\operatorname{pd}_R N$, $\operatorname{fd}_R N$, and $\operatorname{id}_R N$, are the least lengths of projective resolutions, flat resolutions, and injective coresolutions, respectively, with value $\infty$ if no finite one exists. The \emph{graded weak global dimension} and \emph{graded global dimension} of $R$ are
\[
 \begin{aligned}
 \operatorname{w.gl.dim}_{\mathrm{gr}}R
   &=\sup_N\operatorname{fd}_R N,\\
 \operatorname{gl.dim}_{\mathrm{gr}}R
   &=\sup_N\operatorname{pd}_R N
    =\sup_N\operatorname{id}_R N,
 \end{aligned}
\]
where $N$ ranges over all graded left $R$-modules.
\end{definition}

For background on homological dimensions, see Weibel~\cite[Sections~4.1--4.3]{Weibel}. We use the following graded form of the ghost-dimension bounds. Write $R$ for the graded endomorphism ring of the unit and $H(Z)$ for the graded $R$-module of maps from shifts of the unit to $Z$.

\begin{proposition}[Christensen; Hovey--Lockridge]\label{prop:ghost-dimension-bound}
Suppose a small, idempotent-complete rigid symmetric monoidal stable $\infty$-category is thickly generated by compatible invertible grading objects, including suspension. If a composite of $n$ ghosts from $X$ to $Y$ is nonzero, then
\[
 \operatorname{fd}_R H(X),\quad \operatorname{pd}_R H(X),\quad \operatorname{id}_R H(Y)\geq n,
\]
where dimensions are taken in graded $R$-modules.
\end{proposition}
\begin{proof}
Pass to the Ind-completion, where the grading objects are compact generators. The projective and injective bounds follow from Christensen's projective-class argument and its dual \cite[Theorem~3.5 and Proposition~4.7]{Christensen}, as formulated in~\cite[Propositions~1.4--1.5]{HL-dimensions}. If $\operatorname{fd}_R H(X)<n$, the bounds of Hovey--Lockridge \cite[Proposition~2.1 and Theorem~2.6]{HL-ghost} imply that every composite of $n$ ghosts out of $X$ is phantom. Since $X$ is compact, such a composite is zero.
\end{proof}

\Needspace{8\baselineskip}
\begin{proof}[Proof of Corollary~\ref{cor:homological-consequences}]
Take $E=BP$ and $r=1$ in Theorem~\ref{thm:projective-composites}, and choose $m$ so that the $n$-th power of the ghost is nonzero. The resulting finite comodule object is annihilated by $p$. Proposition~\ref{prop:ghost-dimension-bound} gives all three algebraic bounds on its detecting module. By Proposition~\ref{prop:projective-realization}, the corresponding motivic ghost has a finite cellular source in the uncompleted category, also annihilated by $p$, and its $n$-th power remains nonzero. By Lemma~\ref{lem:descent}, this source is already $p$-complete. Applying Proposition~\ref{prop:ghost-dimension-bound} to the thick subcategory generated by the bigraded shifts of the completed sphere gives the motivic bounds. Letting $n$ grow gives infinite graded weak global and global dimensions for both rings.
\end{proof}

\section{Further realizations}\label{sec:fields}

We extend the intrinsic constructions of Section~\ref{sec:motivic} to other fields. Over $\R$ and algebraically closed fields of characteristic zero, the examples are finite over the uncompleted motivic sphere. Over an arbitrary field, at primes invertible in the field, the Chow comparison gives ghosts on objects finite over the Chow-degree-zero module unit.

\subsection{Real-motivic realizations}\label{subsec:real}

\begin{proposition}\label{prop:real-intrinsic}
For every prime $p$ and integers $r,n\geq1$, there are a finite cellular real-motivic spectrum $X_{\R}$, annihilated by $p^r$, and a ghost endomorphism $\alpha_{\R}$ whose first $n$ composition powers have exact additive order $p^r$. There is also an explicit example
\[
 X_{\R}=C(\rho,\tau)/(p,\alpha_1),\qquad \alpha_{\R}=iv_1q,
\]
with sixteen cells over the uncompleted sphere, for which $\alpha_{\R}$ has exact additive order $p$ and square zero.
\end{proposition}
\begin{proof}
Put $C\rho=\cofib(\rho)$, where $\rho:\Sigma^{-1,-1}\widehat\one_{\R}\to\widehat\one_{\R}$ represents $-1\in\R^\times$. At every prime, Behrens--Shah \cite[Corollary~1.9]{BS} give an equivalence
\begin{equation}\label{eq:BS}
 \SH(\C)_{\cell,p}^{\wedge} \simeq\Mod_{C\rho}\bigl(\SH(\R)_{\cell,p}^{\wedge}\bigr)
\end{equation}
preserving bigraded shifts and sending the complex sphere to $C\rho$. Its underlying functor is $\operatorname{Cell}c_*$, where $c:\operatorname{Spec}\C\to\operatorname{Spec}\R$ and $\operatorname{Cell}$ denotes cellularization. Write $C(\rho,\tau)$ for the image of $C\tau$. Transport the complex examples $(X,\alpha)$ of Proposition~\ref{prop:projective-realization} and Theorem~\ref{thm:intrinsic} to $(X_{\R},\alpha_{\R})$. Free--forgetful adjunction identifies their bigraded homotopy groups and induced maps, so $\alpha_{\R}$ is a ghost.

To detect the additive orders, use algebraic cobordism $\mathrm{MGL}$. The spectra are annihilated by $p^r$, so their ordinary smash products with $\mathrm{MGL}$ are $p$-complete by Lemma~\ref{lem:descent}. Cellularity and base change for $\mathrm{MGL}$ and the projection formula \cite[Lemmas~3.7 and~4.5]{BS} give
\[
 \mathrm{MGL}_{\R}\wedge X_{\R} \simeq\operatorname{Cell}c_* \bigl(\mathrm{MGL}_{\C}\wedge X\bigr).
\]
Adjunction and the $\mathrm{MGL}$-homology calculation of Section~\ref{sec:motivic} show that every nonzero multiple of each composition power remains nonzero. Thus forgetting preserves exact additive orders; transport also preserves $\alpha^2=0$ in the explicit example.

For finiteness over the uncompleted sphere, Lemma~\ref{lem:descent} identifies the four-cell spectrum $\one_{\R}/(p^r,\rho)$ with $C\rho/p^r$. Its $\tau$-cofiber models $C(\rho,\tau)/p^r$ with eight cells. The projective-space filtration has $m$ shifts of this spectrum as quotients, hence gives a model with $8m$ cells. Choose $m$ as in Theorem~\ref{thm:projective-composites}. For the explicit example, start with $\one_{\R}/(p,\rho)$ and take the cofibers of the transported lift $\widetilde\alpha_1$ from Section~\ref{subsec:complex} and then $\tau$. This gives sixteen cells in bidegrees
\begin{equation}\label{eq:real-cells}
 (0,0)+\epsilon_1(1,0)+\epsilon_2(0,-1) +\epsilon_3(2p-2,p-1)+\epsilon_4(1,-1), \qquad\epsilon_i\in\{0,1\}.
\end{equation}
Transporting $iv_1q$ gives the displayed map $\alpha_{\R}$.
\end{proof}

For $c\in\F_p^\times$, the map $1+c\alpha_{\R}$ in the sixteen-cell example is a nonidentity self-equivalence of order $p$ acting identically on all bigraded real-motivic homotopy groups.

\subsection{Algebraically closed fields and Chow modules}\label{subsec:fields}
The construction of Section~\ref{sec:motivic} also works over every algebraically closed field of characteristic zero. The Gheorghe--Wang--Xu comparison~\cite[Remark~1.4]{GWX} and the completion comparison~\cite[Theorem~1]{HKO} both hold over these fields. This construction gives an eight-cell spectrum over the uncompleted sphere with a nonzero ghost of exact additive order $p$ and square zero. The same $\mathrm{MGL}$-homology calculation detects its nonvanishing.

There is also a consequence over an arbitrary field after passing to modules over the Chow-degree-zero sphere. Write $\one_{k,c=0}$ for the degree-zero truncation of the motivic sphere in the Chow $t$-structure of~\cite{BKWX}. For a prime $p$ invertible in $k$, consider the $p$-local category
\[
 \Mod^{\cell}_{(\one_{k,c=0})_{(p)}} \bigl(\SH(k)_{(p)}\bigr).
\]
Here the cellular subcategory is generated under colimits by the bigraded shifts of the module unit $(\one_{k,c=0})_{(p)}$; a finite object belongs to their thick subcategory.

Under the algebraic comparison below, this unit corresponds to $BP_*$. Its $p$-complete cellularization over $\C$ is $C\tau$; over $\R$ at $p=2$, it is $C(\rho,\tau)$ \cite[proof of Corollary~1.2 and Section~5.5]{BKWX}. These models have two and four cells, respectively, over the completed motivic sphere. For a general field, the comparison does not supply a finite-cell model of the Chow unit over either the completed or the uncompleted sphere. Finiteness below is relative to the module unit.

\begin{corollary}\label{cor:chow-fields}
Let $k$ be a field, and let $p$ be a prime invertible in $k$. In the cellular module category above, there are a finite object $X$ and a nonzero endomorphism $\alpha:X\to X$ such that
\[
 [\Sigma^{a,b}(\one_{k,c=0})_{(p)},\alpha]_{(\one_{k,c=0})_{(p)}}=0 \qquad(a,b\in\Z).
\]
The object $X$ is constructed from four cells over the module unit, and $\alpha$ has exact additive order $p$ and square zero.
\end{corollary}
\begin{proof}
Bachmann--Kong--Wang--Xu~\cite[Theorem~1.9(2) and Corollary~4.29]{BKWX} identify the cellular module category over $\one_{k,c=0}$, after inverting the exponential characteristic of $k$, with the stable category of even graded $MU_*MU$-comodules. Since $p$ is invertible in $k$, $p$-localization already inverts the exponential characteristic. Localize this symmetric monoidal equivalence at $p$ and use the equivalence between $p$-local $MU_*MU$-comodules and $BP_*BP$-comodules, recalled in~\cite[Section~5.1, before Corollary~5.5]{BKWX}. It preserves the unit, shifts, finite cofibers, and mapping groups, so it carries $M$ and $f$ from Theorem~\ref{thm:main} to $X$ and $\alpha$. The construction of $M$ from four shifts of $BP_*$ gives four cells over the module unit, and the order and composition assertions are preserved.
\end{proof}

\bibliographystyle{plain}
\bibliography{ref}
\end{document}